\documentclass[11pt]{amsart}
\usepackage[dvips]{graphicx,color}
\usepackage{amssymb,amscd,latexsym,epsfig,hyperref}
\usepackage[curve]{xypic}
\usepackage[mathscr]{euscript}
\DeclareFontFamily{OT1}{rsfs}{}
\DeclareFontShape{OT1}{rsfs}{n}{it}{<-> rsfs10}{}
\DeclareMathAlphabet{\curly}{OT1}{rsfs}{n}{it}

\makeatletter
\newcommand{\eqnum}{\refstepcounter{equation}\textup{\tagform@{\theequation}}}
\makeatother

\renewcommand\;{\hspace{.6pt}}

\newcommand\PP{\mathbb P}
\newcommand\LL{\mathbb L}
\newcommand\C{\mathbb C}

\newcommand\N{\mathbb Z}
\newcommand\Z{\mathbb Z}

\newcommand\cO{\mathcal O}
\newcommand\cA{\mathcal A}

\renewcommand\cD{\mathcal D}
\newcommand\cE{\mathcal E}
\newcommand\cF{\mathcal F}

\newcommand\cI{\mathcal I}

\renewcommand\cL{\mathcal L}

\newcommand\cZ{\mathcal Z}

\makeatletter
\newcommand{\so}{\ \ext@arrow 0359\Rightarrowfill@{}{\hspace{3mm}}\ }
\makeatother
\newcommand{\rt}[1]{\xrightarrow{\ #1\ }}
\newcommand\To{\longrightarrow}

\newcommand\into{\hookrightarrow}
\newcommand\INTO{\ \ar@{^(->}[r]<-.2ex>}
\newcommand{\Into}{\ensuremath{\lhook\joinrel\relbar\joinrel\rightarrow}}

\renewcommand\_{^{}_}

\newcommand{\mat}[4]{\left(\begin{array}{cc} \!\!#1 & #2\!\! \\ \!\!#3 & #4\!\!\end{array}\right)}

\newfont{\bigtimesfont}{cmsy10 scaled \magstep5}
\newcommand{\bigtimes}{\mathop{\lower0.9ex\hbox{\bigtimesfont\symbol2}}}
\renewcommand\={\ =\ }

\newcommand\udot{^{\bullet}}
\newcommand{\wwedge}{\mbox{\Large $\wedge$}}

\newcommand\Gr{\operatorname{Gr}}
\newcommand\AJ{\operatorname{AJ}}

\newcommand\rk{\operatorname{rank}}
\newcommand\vir{\operatorname{vir}}
\newcommand\vd{\operatorname{vd}}

\newcommand\coker{\operatorname{coker}}

\newcommand\id{\operatorname{id}}

\newcommand\Hom{\operatorname{Hom}}
\renewcommand\hom{\curly H\!om}
\newcommand\End{\operatorname{End}}
\newcommand\Ext{\operatorname{Ext}}
\newcommand\ext{\curly E\!\;xt}

\newcommand\Pic{\operatorname{Pic}}
\newcommand\Jac{\operatorname{Jac}}

\newcommand\Cone{\operatorname{Cone}}

\newcommand\beq[1]{\begin{equation}\label{#1}}
\newcommand\eeq{\end{equation}}
\newcommand\beqa{\begin{eqnarray*}}
\newcommand\eeqa{\end{eqnarray*}}

\newcommand\arXiv[1]{\href{http://arxiv.org/abs/#1}{arXiv:#1}}
\newcommand\mathAG[1]{\href{http://arxiv.org/abs/math/#1}{math.AG/#1}}

\DeclareRobustCommand{\SkipTocEntry}[3]{}
\makeatletter
\newcommand\@dotsep{4.5}
\def\@tocline#1#2#3#4#5#6#7{\relax
  \ifnum #1>\c@tocdepth % then omit
  \else
    \par \addpenalty\@secpenalty\addvspace{#2}%
    \begingroup \hyphenpenalty\@M
    \@ifempty{#4}{%
      \@tempdima\csname r@tocindent\number#1\endcsname\relax
    }{%
      \@tempdima#4\relax
    }%
    \parindent\z@ \leftskip#3\relax \advance\leftskip\@tempdima\relax
    \rightskip\@pnumwidth plus1em \parfillskip-\@pnumwidth
    #5\leavevmode #6\relax
    \leaders\hbox{$\m@th
      \mkern \@dotsep mu\hbox{.}\mkern \@dotsep mu$}\hfill
    \hbox to\@pnumwidth{\@tocpagenum{#7}}\par
    \nobreak
    \endgroup
  \fi}
\makeatother

\makeatletter \@addtoreset{equation}{section} \makeatother
\renewcommand{\theequation}{\thesection.\arabic{equation}}

\newtheorem{thm}[equation]{Theorem}
\newtheorem{thm*}{Theorem}
\newtheorem{lem}[equation]{Lemma}
\newtheorem{cor}[equation]{Corollary}
\newtheorem{prop}[equation]{Proposition}

\newenvironment{rmk}{\noindent\textbf{Remark}.}{\medskip}
\newenvironment{rmks}{\noindent\textbf{Remarks}.}{\medskip}

\title{Degeneracy loci, virtual cycles and nested Hilbert schemes I}
\author{Amin Gholampour and Richard P. Thomas}

\begin{document}
\maketitle
\begin{abstract} \noindent
Given a map of vector bundles on a smooth variety, consider the deepest degeneracy locus where its rank is smallest. We show it carries a natural perfect obstruction theory whose virtual cycle can be calculated by the Thom-Porteous formula.

We show nested Hilbert schemes of points on surfaces can be expressed as degeneracy loci. We show how to modify the resulting obstruction theories to recover the virtual cycles of Vafa-Witten and reduced local DT theories.

The result computes some Vafa-Witten invariants in terms of Carlsson-Okounkov operators. This proves and extends a conjecture of Gholampour-Sheshmani-Yau and generalises a vanishing result of Carlsson-Okounkov.
\end{abstract}
\tableofcontents \vspace{-9mm}

%%%%%%%%%%%%%%%%%%%%%%%%%%%%%%%%%%%%%%%%%%%%%%%%%%%%%%%%%%%%%%%%%%%%%%%%%%%

\section{Introduction}
The prototype of a scheme $Z$ with \emph{perfect obstruction theory} \cite{BF} is the zero locus of a section of a vector bundle $E$ on a smooth ambient variety $A$. We recall the construction in the next Section.

\emph{All} perfect obstruction theories are \emph{locally} of this form. In the rare situations where this is also true \emph{globally}, the natural virtual cycle \cite{BF} pushes forward to what  we might expect, namely the Euler class of the bundle:
\beq{Euler}
\iota_*[Z]^{\vir}\=c_r(E)\ \in\ A_{\vd}(A).
\eeq
Here $\iota\colon Z\into A$ is the inclusion, $r=\rk E,\ \vd=\dim A-r$ is the virtual dimension of the problem, and $[Z]^{\vir}$ lies in $A_{\vd}(Z)$ or $H_{2\!\vd}(Z)$.

\eqref{Euler} can help in computing integrals over the virtual cycle. Examples include the computation of the number 27 of lines on a cubic surface, numbers of lines and conics on quintic threefolds, and the quantum hyperplane principle. A more relevant example to us is the reduced stable pair computations in \cite{KT}, carried out by writing the moduli space of stable pairs (and its reduced perfect obstruction theory) as the zero locus of a section of a tautological bundle over a certain Hilbert scheme.\medskip

In this paper we study a generalisation of zero loci, namely degeneracy loci. We show these give another prototype of a perfect obstruction theory.\footnote{In fact we prove this by reducing to the model \eqref{model} in a bigger ambient space.}
Again, when this can be done \emph{globally}, it allows us to express integrals over the virtual cycle in terms of integrals over the ambient space, via the Thom-Porteous formula.

So fix a two term complex of vector bundles $E_\bullet=\{E_0\rt\sigma E_1\}$ on a smooth ambient space $A$. Set $n=\dim A,\ r_i=\rk(E_i)$, and denote the $r$th degeneracy locus by
$$
Z_r\ :=\ \big\{x\in A\colon\rk(\sigma|_x)\le r\}.
$$
We work with the smallest $r$ for which $Z:=Z_r$ is nonempty. Our first result is the following, made more precise in Theorem \ref{prop}.

\begin{thm*} Assume $Z_{r-1}=\emptyset$. Then both $h^0(E_\bullet|_Z)=\ker(\sigma|_Z)$ and $h^1(E_\bullet|_Z)=\coker(\sigma|_Z)$ are locally free on $Z:=Z_r$, which inherits a perfect obstruction theory
$$
\Big\{h^1(E_\bullet|_Z)^*\otimes h^0(E_\bullet|_Z)\To\Omega_A|\_Z\Big\}\To\LL_Z.
$$
The push forward of the resulting virtual cycle $[Z]^{\vir}\in A_{n-k}(Z)$ to $A$ is given by the Thom-Porteous formula
$$
\Delta\;_{r_1-r}^{r_0-r}\big(c(E_1-E_0)\big)\ \in\ A_{n-k}(A),
$$
where $k=(r_0-r)(r_1-r)$ and $\Delta^a_b(c):=\det\!\big(c_{b+j-i}\big)_{1\le i,j\le a}\,.$
\end{thm*}

\smallskip\noindent\textbf{Nested Hilbert schemes.} 
Our main application is to the punctual  Hilbert schemes of \emph{nested} subschemes of a fixed projective surface $S$. Full details and notation will be described later; for now for simplicity we restrict attention to the simplest case of the 2-step nested punctual Hilbert scheme
$$
S^{[n_1,n_2]}\ :=\ \big\{I_1\subseteq I_2\subseteq\cO_S\ \colon\ \mathrm{length}\;(\cO_S/I_i)=n_i\big\}.
$$
Now $S^{[n_1,n_2]}$ lies in the ambient space $S^{[n_1]}\times S^{[n_2]}$ as the locus of points $(I_1,I_2)$ for which there is a nonzero map $\Hom_S(I_1,I_2)\ne0$. Thus it can be seen as the degeneracy locus of the complex of vector bundles
\beq{fortnite}
R\hom_\pi(\cI_1,\cI_2) \quad\mathrm{over}\quad S^{[n_1]}\times S^{[n_2]}
\eeq
which, when restricted to the point $(I_1,I_2)$, computes $\Ext^*_S(I_1,I_2)$. When $H^{0,2}(S)=0$ this complex is 2-term, so we can apply the above theory. The resulting perfect obstruction theory on $S^{[n_1,n_2]}$ agrees with that of \cite{GSY1}. In turn this arises in local DT theory \cite{GSY2}, so we can express DT integrals in terms of Chern classes of tautological sheaves over $S^{[n_1]}\times S^{[n_2]}$.

When $H^{0,1}(S)\ne0$ the result is zero; when $H^{0,2}(S)\ne0$ the theory does not apply. So for a general projective surface $S$ we modify the complex $\Ext^*_S(I_1,I_2)$ with $H^1(\cO_S)$ and $H^2(\cO_S)$ terms. The modification is canonical over $S^{[n_1,n_2]}$, recovering the \emph{reduced} version of the local DT deformation theory that arises in the $SU(r)$ Vafa-Witten theory of $S$ \cite{TT1}. \smallskip

\noindent\textbf{Splitting trick.} We would like to extend this modification over the rest of $S^{[n_1]}\times S^{[n_2]}$, so we can apply the Thom-Porteous formula. Such modifications exist locally but \emph{not} globally, so in Section \ref{splitpr} we use a trick reminiscent of the splitting principle in topology, pulling back to a certain bundle over $S^{[n_1]}\times S^{[n_2]}$ where there \emph{is} a canonical modification. This allows us to prove the following (whose notation will be explained more fully in Sections \ref{00}--\ref{kstep}, in particular \eqref{trayce}).

\begin{thm*}
Let $S$ be any smooth projective surface. The $k$-step nested Hilbert scheme $S^{[n_1,\ldots,n_k]}$ can be seen as an intersection of degeneracy loci after pulling back to an affine bundle over $S^{[n_1]}\times\cdots\times S^{[n_k]}$. The resulting perfect obstruction theory $F\udot\to\LL_{S^{[n_1,\ldots,n_k]}}$ has virtual tangent bundle
$$
(F\udot)^\vee\ \cong\ \Big\{T_{S^{[n_1]}}\oplus\cdots\oplus T_{S^{[n_k]}}\To\ext^1_p(\cI_1,\cI_2)\_0\oplus\cdots\oplus\ext^1_p(\cI_{k-1},\cI_k)\_0\Big\},
$$
the same as the one in Vafa-Witten theory \cite{TT1} or ``reduced local DT theory" \cite{GSY1, GSY2}. The virtual cycle
$$
\big[S^{[n_1,\ldots,n_k]}\big]^{\vir}\,\in\,A_{n_1+n_k}\big(S^{[n_1,\ldots,n_k]}\big)
$$
pushes forward to
\beq{CO}
c_{n_1+n_2}\big(R\hom_{\pi}(\cI_1,\cI_2)[1]\big)\cup\cdots\cup c_{n_{k-1}+n_k}\big(R\hom_{\pi}(\cI_{k-1},\cI_k)[1]\big)
\eeq
in $A_{n_1+n_k}\big(S^{[n_1]}\times\cdots\times S^{[n_k]})$.
\end{thm*}

The formula \eqref{CO} for the pushforward of the virtual class was conjectured in \cite{GSY1} for $k=2$ and proved for toric surfaces. It was also shown to be true for more general surfaces when integrated against some natural classes. The classes $c_{n_{i-1}+n_i}\big(R\hom_{\pi}(\cI_{i-1},\cI_i)[1]\big)$, considered as maps $H^*(S^{[n_{i-1}]})\to H^{*+2n_i-2n_{i-1}}(S^{[n_i]})$, are called Carlsson-Okounkov operators. Carlsson-Okounkov \cite{CO} calculate them in terms of Grojnowski-Nakajima operators, and prove vanishing of the higher Chern classes:
\beq{vanz}
c_{n_1+n_2+i}\big(R\hom_{\pi}(\cI_1,\cI_2)[1]\big)\=0, \qquad i>0,
\eeq
by showing the left hand side is a universal expression in Chern numbers of $S$, and that this universal expression vanishes for toric surfaces by a localisation computation. This gives enough relations to prove the universal expression is in fact zero.
In Section \ref{COvan} we reprove the vanishing \eqref{vanz} rather easily and geometrically using the Thom-Porteous formula, as well as the following generalisation.

\begin{thm*}
Let $S$ be any smooth projective surface. For any curve class $\beta\in H_2(S,\Z)$, any Poincar\'e line bundle $\cL\to S\times\Pic_{\beta}(S)$, and any $i>0$,
$$
c_{n_1+n_2+i}\big(\!\;R\pi_*\;\cL-R\hom_{\pi}(\cI_1,\cI_2\otimes\cL)\big)\=0 \quad\mathrm{on}\ S^{[n_1]}\times S^{[n_2]}\times\Pic_\beta(S).
$$
\end{thm*}
\medskip

\noindent\textbf{The other degeneracy loci.}
In the companion paper \cite{GT2} we work with \emph{all} the degeneracy loci $Z_k$. These do not generally admit perfect obstruction theories when $k>r$. However there are natural spaces $\widetilde Z_k\to Z_k$ dominating them which are actually resolutions of their singularities in the transverse case (when all the $Z_k$ have the correct codimension). For this reason we call the $\widetilde Z_k$ ``\emph{virtual resolutions}". Though they are singular in general we show they admit natural perfect obstruction theories and virtual cycles whose pushforwards we can again describe by Chern class formulae.\footnote{Since $\widetilde Z_r\cong Z_r$ the constructions in \cite{GT2} and  this paper coincide when $k=r$.}

In this paper the natural application was to nested \emph{punctual} Hilbert schemes of a smooth surface $S$. In \cite{GT2} the natural application is to nested Hilbert schemes of both points \emph{and curves} in $S$. Fundamentally the difference is the following. Letting $I_1,I_2\subset\cO_S$ be ideal sheaves of 0-dimensional subschemes of $S$, then
\beq{hom12}
\Hom(I_1,I_2)
\eeq
either vanishes, or --- for $I_1\subset I_2$ in the nested Hilbert scheme --- is at most $\C$. Hence $S^{[n_1,n_2]}$ is the degeneracy locus of the complex \eqref{fortnite}. Conversely, when $I_1$ or $I_2$ have divisorial components, \eqref{hom12} can become arbitrarily big, and different elements correspond to different subschemes of $S$. (In the case $I_1=\cO_S(-D)$ and $I_2=\cO_S$, elements correspond --- up to scale --- to divisors in the same linear system as the divisor $D\subset S$.) Therefore the corresponding nested Hilbert scheme \emph{dominates} the degeneracy locus of the complex \eqref{fortnite} but need not equal it. In \cite{GT2} we show it is naturally a virtual resolution of the type $\widetilde Z_k$.

\smallskip\noindent\textbf{Acknowledgements.} 
Artan Sheshmani was part of a good portion of this project, but decided to concentrate on developing virtual fundamental classes in higher rank situations \cite{SY}. We thank him for many useful conversations, as well as Davesh Maulik, Andrei Negut and Andrei Okounkov. A.G. acknowledges partial support from NSF grant DMS-1406788. 

There are some constructions in the literature which are closely related to ours; see for instance \cite[Equation 6.2]{Ne1} for $S=\PP^2$, \cite[Section 2.3]{Ne2} for some special types of nested sheaves, and \cite[Equation (2.23)]{Ne2} for their perfect obstruction theory. Just before posting this paper we became aware of the old notes \cite{MO}, which describe the local DT obstruction theory \cite{GSY2} on the nested Hilbert scheme, and a K-theoretic version of the Carlsson-Okounkov operator on toric surfaces. With hindsight it seems that Okounkov et al probably new of some form of relationship between the virtual class and the Thom-Porteous formula for toric surfaces with $H^{\ge1}(\cO_S)=0$, even if they're too modest to admit it now. 

\medskip\noindent\textbf{Notation.} Given a map $f\colon X\to Y$, we often use the same letter $f$ to denote its basechange by any map $Z\to Y$, i.e. $f\colon X\times_YZ\to Z$. We also sometimes suppress pullback maps $f^*$ on sheaves.

\section{Zero loci} \label{zl}
We start by recalling the standard construction of a perfect obstruction theory, on the zero scheme $Z$ of a section $\sigma$ of a vector bundle $E$ over a smooth ambient space $A$:
\beq{model}
\xymatrix@R=15pt@C=0pt{
& E\dto \\
Z=\sigma^{-1}(0)\ \subset & A.\ar@/^{-2ex}/[u]_\sigma}
\eeq
On $Z$ the derivative of this diagram gives
\beq{model2}
\xymatrix@R=18pt{
E^*|_Z \ar[d]_\sigma\ar[r]^{d\sigma|_Z}& \Omega_A|_Z\ar@{=}[d] \\
I/I^2 \ar[r]^d& \Omega_A|_Z,}
\eeq
where $I\subset\cO_A$ is the ideal sheaf of $Z$ generated by $\sigma$. The bottom row is a representative of the truncated cotangent complex $\LL_Z$ of $Z$; denoting the two-term locally free complex on the top row by $F\udot$ we get a morphism\footnote{\eqref{model} also induces a natural map from $F\udot$ to the \emph{full} cotangent complex of $Z$ \cite[Section 6]{BF}, but we shall not need this.}
\beq{pot}
F\udot\To\LL_Z
\eeq
in $D(\mathrm{Coh}\,Z)$ which induces an isomorphism on $0$th cohomology sheaves $h^0$ and a surjection on $h^{-1}$. This data is called a \emph{perfect obstruction theory} \cite{BF} on $Z$, and induces a virtual cycle
$$
[Z]^{\vir}\ \in\ A_{\vd}(Z)\To H_{2\!\vd}(Z)
$$
satisfying natural properties. Here $H$ denotes Borel-Moore homology, and $\vd:=\dim A-\rk E$ is the \emph{virtual dimension} of the perfect obstruction theory.

\section{Degeneracy loci}\label{degloc}
We work on a smooth complex quasi-projective variety $A$ with a map 
$$
E_0\rt{\sigma}E_1
$$
between vector bundles of ranks $r_0$ and $r_1$. We denote by
\beq{Zkay}
Z_k\ \subset\ A
\eeq
the degeneracy locus where $\rk(\sigma)$ drops to $\le k$. This has a scheme structure defined by the vanishing of the $(k+1)\times(k+1)$ minors of $\sigma$, i.e. of
\beq{wwedge}
\wwedge^{k+1}\sigma\,\colon\,\wwedge^{k+1}E_0\To\wwedge^{k+1}E_1.
\eeq
The $Z_k$ can be characterised by the rank of the cokernel of $\sigma$ over them \cite[Section 20.2]{Ei}. In Section \ref{arb} we will need a characterisation in terms of the kernel. Though this does not basechange well, it works for the smallest $Z_k$.

That is, let $r$ denote the minimal rank of $\sigma$, so that $Z_{r-1}=\emptyset$, and set $Z:=Z_r$.
This is the largest subscheme of $A$ on which $\ker\sigma|_Z$ is locally free of rank $r_0-r$:

\begin{lem} \label{h0base} For a map of schemes $f\colon T\to A$, the following are equivalent.
\begin{enumerate}
\item $f$ factors through $Z=Z_r\subset A$,
\item $\ker\big(f^*\sigma\colon f^*E_0\to f^*E_1\big)$ is a rank $r_0-r$ subbundle of $f^*E_0$,
\item $\ker\big(f^*\sigma\colon f^*E_0\to f^*E_1\big)$ has a locally free subsheaf of rank $r_0-r$.
\end{enumerate}
\end{lem}

\begin{proof}
If $f$ factors through $Z$ then $\wwedge^{r+1}f^*\sigma\cong f^*\wwedge^{r+1}\sigma|_Z\equiv0$. Since $Z_{r-1}=\emptyset$ it follows from \cite[Proposition 20.8]{Ei} that $\coker f^*\sigma$ is locally free of rank $r_1-r$. Thus $\ker f^*\sigma$ is a rank $r_0-r$ subbundle of $f^*E_0$. This proves $(1)\so(2)\so(3)$.

For $(3)\so(1)$, we
suppose the kernel $K$ of $f^*E_0\to f^*E_1$ contains a locally free subsheaf of rank $r_0-r$. Therefore 
the rank of $f^*\sigma$ on the generic point of $T$ is $\le r$, and thus in fact equal to $r$ since we are assuming it drops no lower. In particular, $\coker(f^*\sigma)$ is a rank $r_1-r$ sheaf.

By lower semi-continuity of rank, $f^*\sigma|_t$ is of rank $\le r$ for any closed point $t\in T$, so, by our assumption on $r$ again, it is equal to $r$. Combined with the exact sequence
\beq{kerr}
f^*E_0|\_t\rt{\sigma|_t}f^*E_1|\_t\To(\coker f^*\sigma)|\_t\To0,
\eeq
i.e. the fact that $\coker(f^*\sigma|_t)=(\coker f^*\sigma)|_t$, this shows that $(\coker f^*\sigma)|_t$ has dimension $r_1-r$ for every closed point $t$. Therefore $\coker f^*\sigma$ is locally free of rank $r_1-r$ by the Nakayama lemma. This implies that $\ker f^*\sigma$ is a rank $r_0-r$ subbundle (rather than just a locally free subsheaf) of $f^*E_0$.

In particular $f^*E_0/K$ is locally free of rank $r$, so $\wwedge^{r+1}(f^*E_0/K)=0$. But
$$
f^*\wwedge^{r+1}\sigma\=\wwedge^{r+1}f^*\sigma\,\colon\ \wwedge^{r+1}f^*E_0\To\wwedge^{r+1}f^*E_1
$$
factors through $\wwedge^{r+1}(f^*E_0/K)$, so it is also zero. That is, $f$ factors through the zero scheme $Z\left(\wwedge^{r+1}\sigma\right)=Z_r$ of $\wwedge^{r+1}\sigma$.
\end{proof}

So $\sigma|_Z$ has rank precisely $r$, and its
kernel $h^0:=h^0(E_\bullet|_Z)$ and cokernel $h^1:=h^1(E_\bullet|_Z)$ are \emph{vector bundles} on $Z$ of rank $r_0-r$ and $r_1-r$ respectively,
\beq{h*}
0\To h^0\To E_0|_Z\rt{\sigma|_Z}E_1|_Z\To h^1\To0.
\eeq
For instance if $r=r_0-1$ then $\sigma$ is generically injective (and globally injective as a map of coherent sheaves) and $Z$ is the locus where it fails to be injective as a map of bundles. Its kernel is a line bundle over $Z$. If $E_0=\cO_A$ then $Z$ is the zero locus of $\sigma$ and we are back in the setting of Section \ref{zl}.

\begin{thm} \label{prop} The degeneracy locus $Z=Z_r$ inherits a 2-term perfect obstruction theory
$$
\big\{(h^1)^*\otimes h^0\To\Omega_A|_Z\big\}\To\LL_Z.
$$
The push forward of the resulting virtual cycle $[Z]^{\vir}\in A_{n-k}(Z)$ to $A$ is given by the Thom-Porteous formula
$$
\Delta\;_{r_1-r}^{r_0-r}\big(c(E_1-E_0)\big)\ \in\ A_{n-k}(A).
$$
Here $n=\dim A,\ k=(r_0-r)(r_1-r)$ and $\Delta\;^a_b(c):=\det\!\big(c_{b+j-i}\big)_{1\le i,j\le a}$\,.
\end{thm}

\begin{proof}
We work on the relative Grassmannian of $(r_0-r)$-dimensional subspaces of $E_0$,
$$
\Gr\,:=\,\Gr(r_0-r,E_0)\rt{q}A
$$
with universal subbundle $U\into q^*E_0$. Composing with $q^*\sigma$ gives a section
\beq{tilda}
\widetilde\sigma\ \in\ \Gamma(U^*\otimes q^*E_1).
\eeq
\textbf{Claim 1.} The zero locus $Z(\widetilde\sigma)\subset\Gr$ is isomorphic to $Z\subset A$ under the restriction $\overline{q}\colon Z(\widetilde\sigma)\to A$ of the projection $q\colon\!\Gr\to A$. \smallskip

At the level of closed points this is obvious: for $x\in A$
\begin{align*}
x\in Z&\iff\rk(\sigma|_x)=r \\
&\iff\rk(\ker(\sigma_x))=r_0-r \\
&\iff(E_0)|_x\mathrm{\ has\ a\ unique\ }(r_0-r)\mathrm{-dimensional\ subspace} \\ &\hspace{36mm}U_x
=\ker(\sigma_x)\mathrm{\ on\ which\ }\sigma|_x\mathrm{\ vanishes}\\
&\iff U_x\ \mathrm{is\ the\ unique\ point\ of\ }Z(\widetilde\sigma)\cap q^{-1}\{x\}.
\end{align*}
So $\overline{q}$ maps $Z(\widetilde\sigma)$ bijectively to $Z\subset A$. To see it maps scheme theoretically, note that, by construction, the composition
$$
U\Into q^*E_0\rt{q^*\sigma}q^*E_1
$$
is zero over $Z(\widetilde\sigma)$, so $\ker(\overline{q}^*\sigma)$ contains a locally free sheaf $U|_{Z(\widetilde\sigma)}$ of rank $r_0-r$. Thus $\overline{q}$ factors through $Z\subset A$ by Lemma \ref{h0base}.

By Lemma \ref{h0base} again, $\ker(\sigma|_Z)$ is a rank $r_0-r$ subbundle of $E_0$. Its classifying map $Z\to\Gr(r_0-r,E_0)$ has image in $Z(\widetilde\sigma)$ and clearly defines a right inverse to $\overline{q}\colon Z(\widetilde\sigma)\to Z$.
So to prove that $\overline{q}$ is an isomorphism to $Z$ we need only show that the inverse image $\overline q^{\,-1}\{x\}$ of any closed point $x\in Z$ is a closed point of $Z(\widetilde\sigma)$.

Given a rank $r$ linear map $\Sigma\colon V\to W$ between vector space of dimensions $r_0,r_1$, an elementary calculation show that the composition
$$
U\Into V\otimes\cO\rt{\Sigma}W\otimes\cO
$$
on the Grassmannian $\Gr(r_0-r,V)$ cuts out the \emph{reduced} point $\big[\!\;\ker\Sigma\subset V\big]\in\Gr(r_0-r,V)$. Applying this to $\Sigma=\sigma|_x$ proves Claim 1. \smallskip

\noindent\textbf{Perfect obstruction theory.} Since $Z\cong Z(\widetilde\sigma)$ is cut out of $\Gr$ by $\widetilde\sigma\in\Gamma(U^*\otimes q^*E_1)$, it inherits the standard perfect obstruction theory \eqref{model2}, i.e.
\beq{mod3}
U\otimes q^*E_1^*|_{Z(\widetilde\sigma)}\rt{d\;\widetilde\sigma|_{Z(\widetilde\sigma)}}\Omega_{\Gr}|_{Z(\widetilde\sigma)}
\eeq
mapping to $\LL_{Z(\widetilde\sigma)}=\LL_Z$. Now \eqref{mod3} fits into a diagram
\beq{dg}
\xymatrix@R=18pt{
U\big|_Z\otimes(h^1)^* \ar[d]\ar[r] & q^*\Omega_A\big|_{Z(\widetilde\sigma)} \ar[d] \\
U\otimes E_1^*\big|_Z \ar[d]_{\id_U\otimes\!}^{\!\sigma^*}\ar[r]^{d\;\widetilde\sigma|_{Z(\widetilde\sigma)}} & \Omega_{\Gr}\big|_{Z(\widetilde\sigma)} \ar[d] \\
U|_Z\otimes\big(E_0|_Z\big/\ker\sigma\big)^* \ar@{=}[r]& 
\Omega_{\Gr/A}\big|_{Z(\widetilde\sigma)\,,}}
\eeq
with left hand column the short exact sequence $U|_Z\otimes\,$\eqref{h*}${}^*$, and right hand column the natural short exact sequence of the fibration $\Gr\to A$. The bottom equality is dual to the standard identification $T_{\Gr/A}\cong\hom(U,E_0/U)$. 

\medskip
Assuming \eqref{dg} is commutative for now, we can consider it as providing a quasi-isomorphism between the top row and the middle row (which is \eqref{mod3}). Hence the perfect obstruction theory \eqref{mod3} is
$$
h^0\otimes(h^1)^*\To\Omega_A|_Z,
$$
as claimed. Just as in \eqref{Euler}, the pushforward of the resulting virtual cycle to $\Gr$ is the Euler class $c_{(r_0-r)r_1}(U^*\otimes q^*E_1)$. Pushing this down to $A$ gives the pushforward of $[Z]^{\vir}$ to $A$, by the commutativity of the diagram
$$
\xymatrix@R=18pt{
Z(\widetilde\sigma) \INTO\ar@{=}[d] & \Gr \ar[d]^ q \\
Z \INTO & A.\!}
$$
But pushing forward $c_{(r_0-r)r_1}(U^*\otimes q^*E_1)$ to $A$ gives $\Delta\;_{r_1-r}^{r_0-r}\big(c(E_1-E_2)\big)$ by \cite[Theorem 14.4]{Fu}. So we are left to prove \medskip

\noindent\textbf{Claim 2.} The diagram \eqref{dg} is commutative. \smallskip

We need only show that the lower square of \eqref{dg} commutes; the upper one is then induced from it. Let $\Gr\_Z:=\Gr\times\_{\!A}\;Z$ and observe $Z(\widetilde\sigma)\subset\Gr\_Z$, with ideal sheaf $I$ say. We let
$$
2Z\ \Into\ \Gr_Z
$$
be its scheme-theoretic doubling with ideal sheaf $I^2$. Let $p:= q|_{2Z}$ be the induced projection $2Z\to Z$ and consider the maps
\beq{ars}
U|\_{2Z}\ \Into\ (q^*E_0)|\_{2Z}\,\cong\,p^*(E_0|_Z) \To p^*(E_0/U|_Z)\rt{\sigma|_Z}p^*(E_1|_Z).
\eeq
The final arrow is constructed from $\sigma|_Z\colon E_0|_Z\to E_1|_Z$ by recalling that $U|_Z\cong\ker(\sigma|_Z)$.

%$$
%\xymatrix@R=18pt@C=12pt{
%U\otimes\Omega_{\Gr\_{Z\!}/Z} \ar[d]&& (E_1/U)\otimes\Omega_{\Gr\_{Z\!}/Z} \ar[d] \\
%\ \ U|\_{2Z} \ar[d]\INTO& (q^*E_1)|\_{2Z}\,\cong\,p^*(E_1|_Z) \ar[r]&
%p^*(E_1/U|_Z) \ar[d]\ar[r]^-{\sigma}& p^*(E_2|_Z). \\
%U|_Z && (E_1/U)|\_Z}
%$$
%The two vertical columns are the obvious short exact sequences arising from tensoring with $0\to\Omega_{\Gr\_{Z\!}/Z}\to\cO_{2Z}\to\cO_Z\to0$. The composition $U|\_{2Z}\to p^*(E_1/U|_Z)$ vanishes, so lifts to a map 
%$U|\_{2Z}\to(E_1/U)\otimes\Omega_{\Gr\_{Z\!}/Z}$ which vanishes on $U\otimes\Omega_{\Gr\_{Z\!}/Z}$. Thus it descends to a map
%$$
%U|\_Z\To(E_1/U)\otimes\Omega_{\Gr\_{Z\!}/Z}\,.
%$$
The composition of the first two arrows of \eqref{ars} is a section of $U^*|\_{2Z}\otimes p^*(E_0/U|_Z)$ on $2Z$ which vanishes on $Z$. Since the ideal of $Z\subset2Z$ is $\Omega_{\Gr\_{Z\!}/Z}$ it is a section of $(U|_Z)^*\otimes(E_0/U|_Z)\otimes\Omega_{\Gr\_{Z\!}/Z}$.
This is precisely the (adjoint of) the standard description of the isomorphism
$$
U|\_Z\otimes(E_0/U)|_Z^*\ \cong\ \Omega_{\Gr\_{Z\!}/Z}\,,
$$
i.e. the bottom row of \eqref{dg}.

Since $p^*(E_1|_Z)=(q^*E_1)|_{2Z}$, the composition of all the arrows in \eqref{ars} is just $\widetilde\sigma|_{2Z}$. It vanishes on $Z$, defining the section $[d\;\widetilde\sigma|_Z]$ of
$$
(U|_Z)^*\otimes E_1|_Z\otimes I/I^2\ \cong\ \hom\big(U\otimes E_1^*|\_Z,\Omega_{\Gr/A}|\_Z\big)
$$
which defines the central arrow of \eqref{dg}.
Thus \eqref{dg} commutes.
\end{proof}

\subsection{Higher Thom-Porteous formula}
When $r_0-r=1$, so the sheaf $h^0$ is a line bundle on the degeneracy locus $Z$, the following ``higher" Thom-Porteous formula will be useful later. Let $\iota\colon Z\into A$ denote the inclusion.

\begin{prop} \label{simplah}
If $r_0-r=1$ then the Thom-Porteous formula becomes
$$
\iota_*[Z]^{\vir}\=c_{r_1-r_0+1}(E_1-E_0)
$$
in $A_{n+r-r_1}(A)$, and for any $i\ge0$ we have the following extension to higher Chern classes:
\beq{higher}
\iota_*\Big(c_1\big((h^0)^*\big)^i\cap[Z]^{\vir}\Big)\=c_{r_1-r_0+1+i\;}(E_1-E_0).
\eeq
\end{prop}

\begin{proof}
The first part follows from the simplification
$$
\Delta\;_b^a\big(c(\ \cdot\ )\big)\=
c\_b(\ \cdot\ )
$$
when $a=r_0-r=1$.

For the second part, recall from \eqref{tilda} that $Z$ is cut out of $\PP(E_0)\rt{q}A$ by the vanishing of the composition
$$
\cO_{\PP(E_0)}(-1)\Into q^*E_0\rt{q^*\sigma}q^*E_1.
$$
Moreover, over this copy of $Z$, we see that the kernel $h^0$ of $E_0\to E_1$ is $\cO_{\PP(E_0)}(-1)$. Therefore, denoting Segre classes by $s_i$, we have
\begin{align*}
\iota_*\Big(c_1\big((h^0)^*\big)^i&\cap\big[Z\big]^{\vir}\Big)\=q_*\!\left(c_1\big(\cO_{\PP(E_0)}(1)\big)^i\cup c_{r\_1}(q^*E_1(1))\right) \nonumber \\
&\=q_*\left(c_1\big(\cO_{\PP(E_0)}(1)\big)^i\cup \sum_{j=0}^{r_1}c_j(q^*E_1)\cup c_1\big(\cO_{\PP(E_0)}(1)\big)^{r_1-j}\right) \nonumber \\
&\=\sum_{j=0}^{r_1}q_*\left(c_1\big(\cO_{\PP(E_0)}(1)\big)^{i+r_1-j}\cup q^*c_j(E_1)\right) \nonumber \\
&\=\sum_{j=0}^{r_1}s_{i+r_1-j-r_0+1}(E_0)\cap c_j(E_1) \nonumber \\
&\=c_{r_1-r_0+i+1}(E_1-E_0). \qedhere
\end{align*}
\end{proof}

Working throughout this Section with $\sigma^*\colon E_1^*\to E_0^*$ instead of $\sigma\colon E_0\to E_1$ gives the same results, up to some reindexing of notation.

\section{Jumping loci of direct image sheaves}
Suppose $f\colon X\to Y$ is morphism of projective schemes, with $Y$ smooth. Fix either a coherent sheaf $\cF$ on $X$ which is \emph{flat over $Y$}, or a perfect complex $\cF$ on $X$ and assume that \emph{$X$ is flat over $Y$.}

We assume that the cohomologies of $\cF$ on any closed fibre $X_y,\ y\in Y$, are concentrated in only two adjacent degrees $i,\,i+1$. Let $a$ denote the maximal dimension of $h^i(X_y,\cF_y)$ as $y$ varies throughout $Y$. That is, we assume there exists $i\in\N$ such that
\begin{itemize}
\item $h^j(X_y,\cF_y)\=0 \quad \forall j\not\in\{i,i+1\},\ \forall y\in Y$,
\item $h^i(X_y,\cF_y)\ \le\ a \quad \forall y\in Y$.
\end{itemize}
It follows that $h^{i+1}(X_y,\cF_y)$ has maximal dimension $b:=a-(-1)^i\chi(\cF_y)$.

Now $Rf_*\;\cF$ is a perfect complex on $Y$ which,
by basechange and the Nakayama lemma, can be trimmed to be a 2-term complex of locally free sheaves
$$
Rf_*\,\cF\ \simeq\ \big\{E_i\to E_{i+1}\big\}
$$
in degrees $i$ and $i+1$. On restriction to the maximal degeneracy locus
$$
Z_a\ :=\ \big\{y\in Y\colon h^i(X_y,\cF_y)=a\big\}\ \subset\ Y
$$
it has kernel of rank $a$. (Note this labelling convention differs slightly from \eqref{Zkay}.) Let $X_Z:=X\times_YZ$ and $\bar f:=f|_{X_Z}$. By \eqref{wwedge} and Theorem \ref{prop} we deduce the following.

\begin{prop} \label{prop2}
The maximal jumping locus $Z=Z_a$ has a natural scheme structure and perfect obstruction theory
$$
\Big\{\big(R^{i+1}\bar f_*\;\cF\big)^*\otimes R^i\bar f_*\;\cF\To\Omega_Y|_Z\Big\}\To\LL_Z,
$$
with the $R^j\bar f_*\;\cF$ locally free. The resulting virtual cycle
$$
[Z]^{\vir}\,\in\,A_d(Z), \quad d:=\dim Y-ab,
$$
when pushed forward to $Y$, is given by
$$
\Delta\;_b^a\big(c(Rf_*\cF[i+1])\big)\ \in\ A_d(Y). \vspace{-6mm}
$$
\ \hfill$\square$
\end{prop}

The result can also be applied to jump loci of relative Ext sheaves (the cohomology sheaves of $R\hom_f(A,B):=Rf_*R\hom(A,B)$) by setting $\cF:=R\hom(A,B)$. We shall use this on punctual Hilbert schemes next.

\section{Nested Hilbert schemes on surfaces with $b_1=0=p_g$}\label{00}

Given positive integers $n_1\ge n_2\ge\cdots\ge n_k$, the $k$-step nested punctual Hilbert scheme of $S$ is, as a set,
\beqa
S^{[n_1,n_2,\ldots,n_k]} &\!:=& \big\{S\supseteq Z_1\supseteq Z_2\supseteq\cdots\supseteq Z_k\ \colon\ \mathrm{length}(Z_i)=n_i\big\} \\
&=& \big\{I_1\subseteq I_2\subseteq\cdots\subseteq I_k\subset\cO_S \colon\ \mathrm{length}(\cO_S/I_i)=n_i\big\}.
\eeqa
As a scheme it represents the functor which takes any base scheme $B$ to the set of ideals $\cI_1\subseteq\cI_2\subseteq\cdots\subseteq\cI_k\subset\cO_{S\times B}$, flat over $B$, such that the restriction of $\cI_i$ to any closed fibre $S\times\{b\}$ has colength $n_i$.

For simplicity we restrict to $k=2$ for now; we will return to general $k$  in Section \ref{kstep}. \medskip

Let $S$ be a smooth complex projective surface with (for now) $h^{0,1}(S)=0=h^{0,2}(S)$, and fix integers $n_1\ge n_2$. Over
$$
S^{[n_1]}\times S^{[n_2]}\times S\rt{\pi}S^{[n_1]}\times S^{[n_2]}
$$
we have the two universal subschemes $\cZ_1,\,\cZ_2$ and their ideal sheaves $\cI_1,\,\cI_2$. We will apply Proposition \ref{prop2} to the perfect complex
$$
R\hom_\pi(\cI_1,\cI_2)\ :=\ R\pi_*R\hom(\cI_1,\cI_2)
$$
on $S^{[n_1]}\times S^{[n_2]}$. Over the closed point $(I_1,I_2)\in S^{[n_1]}\times S^{[n_2]}$ we have
\beq{uro}
\Ext^i(I_1,I_2)\=0, \qquad i\ne0,1,
\eeq
by Serre duality. Moreover
\beq{homjp}
\Hom(I_1,I_2)\=\left\{\!\!
\begin{array}{ll} 0 & Z_1\not\supseteq Z_2, \\ 
\C & Z_1\supseteq Z_2 \end{array} \right.
\eeq
is generically zero and jumps by 1 (but never more) over the nested Hilbert scheme
\beq{nst}
S^{[n_1,n_2]}\ :=\ \big\{Z_2\subseteq Z_1\subset S,\ \mathrm{length}(Z_i)=n_i\big\},
\eeq
at least set-theoretically. Despite our usual notational conventions (to denote $\pi$ basechanged by $S^{[n_1,n_2]}\into S^{[n_1]}\times S^{[n_2]}$ also by $\pi$) we reserve
$$
p\colon S^{[n_1,n_2]}\times S\To S^{[n_1,n_2]}
$$
for the obvious projection. Since $\cI_1,\cI_2$ are flat over $S^{[n_1]}\times S^{[n_2]}$ they restrict to ideal sheaves over $S^{[n_1,n_2]}$; we denote them by the same letters.

\begin{prop} \label{nest0}
If $h^{0,1}(S)=0=h^{0,2}(S)$ then the 2-step nested Hilbert scheme $S^{[n_1,n_2]}$ carries a perfect obstruction theory
\beq{nhpot}
\Big(\big(\ext^1_p(\cI_1,\cI_2)\big)^*\To\Omega_{S^{[n_1]}\times S^{[n_2]}}\big|_{S^{[n_1,n_2]}}\Big)\To\LL_{S^{[n_1,n_2]}}
\eeq
and virtual cycle
$$
\big[S^{[n_1,n_2]}\big]^{\vir}\,\in\,A_{n_1+n_2}\big(S^{[n_1,n_2]}\big).
$$
Its pushforward to $S^{[n_1]}\times S^{[n_2]}$ is given by
\beq{pfwd}
c_{n_1+n_2}\big(R\hom_{\pi}(\cI_1,\cI_2)[1]\big)\ \in\ A_{n_1+n_2}\big(S^{[n_1]}\times S^{[n_2]}\big).
\eeq
\end{prop}

\begin{proof}
By \eqref{homjp} we may apply Proposition \ref{prop2} to the degeneracy locus $Z$ of $R\hom_\pi(\cI_1,\cI_2)$ by setting $\cF=R\hom(\cI_1,\cI_2)$. By \eqref{uro} and the Nakayama lemma $\cF$ is quasi-isomorphic to a 2-term complex of vector bundles.

As sets $Z\cong S^{[n_1,n_2]}$ by \eqref{homjp}. Over the degeneracy locus $Z$ we have the exact sequence \eqref{h*} with $h^0$ a rank one locally free sheaf, i.e. a line bundle $L$. Thus over $Z\times S$ we obtain a map
$$
\cI_1\otimes p^*L\To \cI_2
$$
which is nonzero on any fibre of $p$. Taking determinants or double duals shows that $L$ is trivial, $h^0\cong\cO_{S^{[n_1,n_2]}}$, and we get a map $\cI_1\to \cI_2$ whose classifying map gives a morphism $Z\to S^{[n_1,n_2]}$.

Conversely, since $p_*\;\hom(\cI_1,\cI_2)=\cO$ over $S^{[n_1,n_2]}$, the latter lies in the degeneracy locus of $R\hom_\pi(\cI_1,\cI_2)$, i.e. $S^{[n_1,n_2]}\subset Z$. It is clear these two maps are inverses.

The rest follows from Proposition \ref{prop2}, simplified as in Proposition \ref{simplah}, and the fact that $h^0\cong\cO_{S^{[n_1,n_2]}}$.
\end{proof}

\begin{rmks} 
In Theorem \ref{nest} we will identify our virtual cycle with that of \cite{GSY1}. The formula \eqref{pfwd} for the pushforward of this cycle was conjectured in \cite{GSY1}, proved there for toric surfaces, and shown to be true for more general surfaces when integrated against some natural classes.

From \eqref{dg} one can work out that the dual of the first arrow in \eqref{nhpot} is
$$
\ext^1_p(\cI_1,\cI_1)\oplus\ext^1_p(\cI_2,\cI_2)\rt{(\iota,\,-\iota^*)}\ext^1_p(\cI_1,\cI_2),
$$
where $\iota\colon\cI_1\to\cI_2$ is the natural inclusion. This complex is therefore the virtual tangent bundle of our perfect obstruction theory on $S^{[n_1,n_2]}$.
\end{rmks}

\section{Removing $H^1(\cO_S)$ and $H^2(\cO_S)$ on arbitrary surfaces} \label{arb}
When $h^{0,1}(S)>0$ the virtual cycle constructed in the last section becomes zero due to a trivial $H^1(\cO_S)$ piece in its obstruction sheaf. And when $h^{0,2}(S)>0$ the perfect complex $R\hom_\pi(\cI_1,\cI_2)$ over $S^{[n_1]}\times S^{[n_2]}$ becomes \emph{3-term}, as it has nonzero $h^2=\ext^2_\pi(\cI_1,\cI_2)$.

So we want to modify $R\hom_\pi(\cI_1,\cI_2)$ with $H^1(\cO_S)$ and $H^2(\cO_S)$ terms. The correct geometric way to do this is to take the product of our ambient space $S^{[n_1]}\times S^{[n_2]}$ with $\Jac(S)$ --- we do this in Section \ref{jack} when $h^{0,2}(S)=0$.\footnote{When $h^{0,2}(S)>0$ one should do the same with the \emph{derived scheme} $\Jac(S)$ with nonzero obstruction bundle $H^2(\cO_S)\otimes\cO$. We don't go this far.} In this Section we use a more \emph{ad hoc} fix which is less geometric but appears to give stronger results.

To describe it, consider the natural composition
$$
\xymatrix@C=0pt@R=16pt{
H^2(\cO_S)\otimes\_\C\cO_{S^{[n_1]}\times S^{[n_2]}} \ar[rrd]& \cong\ R^2\pi_*\;\cO\ \cong\ R^2\pi_*\;\cI_2\ = & \!\ext^2_\pi(\cO,\cI_2)\  \ar[d]^{\iota^*_1} \\
&&\!\ext^2_\pi(\cI_1,\cI_2),}\vspace{-9mm}
$$
\beq{ont}
\eeq
induced by $\iota\_1\colon\cI_1\to\cO_{S^{[n_1]}\times S^{[n_2]}\times S}$\;. Since $\ext^3_\pi(\cO/\cI_1,\cI_2)=0$ (because $\pi$ has relative dimension 2) the composition \eqref{ont} is \emph{surjective}. Therefore, if there was a lifting
\beq{find}
\xymatrix@R=18pt{
H^2(\cO_S)\otimes\cO[-2]\ar@{->>}[dr]\ar@{.>}[d] \\
R\hom_\pi(\cI_1,\cI_2) \ar[r]& \ext^2_\pi(\cI_1,\cI_2)[-2],\!}
\eeq
then the cone on the dotted arrow \eqref{find} would have no $h^2$ and so would be quasi-isomorphic to a 2-term complex of vector bundles. So we could replace $R\hom_\pi(\cI_1,\cI_2)$ by this cone: they have the same $h^0$ jumping locus $S^{[n_1,n_2]}$ (this is proved in Lemma \ref{gw}; it is not true for the $h^{\ge1}$ jumping loci, however) and the same Chern classes. Assuming we could find a similar lift for $H^1(\cO_S)\otimes\cO[-1]$ as well, applying Theorem \ref{prop} to the cone would give the following.

\begin{thm} \label{aim}
Let $S$ be any smooth projective surface. The 2-step nested Hilbert scheme $S^{[n_1,n_2]}$ carries a natural\,\footnote{Naturality will follow from the fact that the lift \eqref{find} is canonical on restriction to $S^{[n_1,n_2]}\subset S^{[n_1]}\times S^{[n_2]}$; see \eqref{rite}.} perfect obstruction theory
and virtual cycle
$$
\big[S^{[n_1,n_2]}\big]^{\vir}\,\in\,A_{n_1+n_2}\big(S^{[n_1,n_2]}\big)
$$
whose pushforward to $S^{[n_1]}\times S^{[n_2]}$ is
$c_{n_1+n_2}\big(R\hom_{\pi}(\cI_1,\cI_2)[1]\big).$
\end{thm}

Unfortunately the lifting \eqref{find} does not exist in general, so to prove the Theorem we will use a trick borrowed from the splitting principle in topology: we pull back to a bigger space $\cA\to S^{[n_1]}\times S^{[n_2]}$ where there is such a splitting, then show the passage does not destroy any information.

For the rest of this Section we carry this out, dealing similarly with $H^1(\cO_S)$ at the same time.
\medskip

We denote by $R^{\ge1}\pi_*\;\cO$ the truncation $\tau^{\ge1}R\pi_*\;\cO$. Choosing once and for all a splitting of $R\Gamma(\cO_S)$ into its cohomologies induces a splitting
\beq{1stsplit}
R^{\ge1}\pi_*\;\cO\ \cong\ H^1[-1]\ \oplus\ H^2[-2],
\eeq
where
$$
H^i\ \colon=\ H^i(\cO_S)\otimes\cO_{S^{[n_1]}\times S^{[n_2]}}
$$
is the trivial vector bundle of rank $h^{0,i}(S)$ over $S^{[n_1]}\times S^{[n_2]}$. As described above, we wish to map this to $R\hom_{\pi}(\cI_1,\cI_2)$ in an appropriate way, which we will do by factoring through the map
\beq{amic}
\iota_1^*\,\colon\ R\pi_*\;\cI_2\To R\hom_{\pi}(\cI_1,\cI_2)
\eeq
induced by $\iota\_1\colon\cI_1\to\cO$. We relate $R\pi_*\;\cI_2$ and $R^{\ge1}\pi_*\;\cO$ by the commutative diagram of exact triangles
$$
\xymatrix@=18pt{
& \cO \ar[d]_{h^0}\ar@{=}[r]& \cO \ar[d] \\
R\pi_*\;\cI_2 \ar@{=}[d]\ar[r]& R\pi_*\;\cO \ar[d]\ar[r]& \pi_*\big(\cO/\cI_2\big) \ar[d] \\
R\pi_*\;\cI_2 \ar[r]& R^{\ge1}\pi_*\;\cO \ar[r]& \cO^{[n_2]}/\cO.\!}
$$
Here $\cO^{[n_2]}:=\pi_*(\cO/\cI_2)$ is the tautological vector bundle, and the top two rows induce the bottom one. This gives the exact triangle
\beq{split?}
\xymatrix{
\cO^{[n_2]}/\cO\,[-1] \ar[r]& R\pi_*\;\cI_2 \ar@<.5ex>[r]& R^{\ge1}\pi_*\;\cO \ar@{.>}@<.5ex>[l]}
\eeq
which we want to split (to then compose with \eqref{amic}).  To write this more explicitly, we split $R^{\ge1}\pi_*\;\cO$ by \eqref{1stsplit} and fix a 2-term locally free resolution $F_1\to F_2$ of $R\pi_*\;\cI_2$, with $F_i$ in degree $i$. Then \eqref{split?} gives
\beq{dig}
\xymatrix{
& \cO^{[n_2]}/\cO \\
0 \ar[r]& R^1\pi_*\;\cI_2 \ar[r]^-{h^1}\ar@{->>}[d]_{\iota\_2}\ar@{<-^)}[u]-U& F_1 \ar[r]& F_2 \ar[r]^-{h^2}& R^2\pi_*\;\cI_2 \ar@{=}[d]\ar[r]& 0 \\
& H^1 \ar@{.>}@/^{-2ex}/[u]_{\phi_1} &&& H^2,\!\!\! \ar@{.>}[ul]^-{\phi_2}}
\eeq
where $\iota\_2\colon\cI_2\to\cO$ and the left hand column is a short exact sequence.
Choices of splittings $\phi_1,\,\phi_2$ would induce a splitting of \eqref{split?}. 

Since the $H^i$ are \emph{free}, splittings $(\phi_1,\phi_2)$ of \eqref{dig} exist \emph{locally}. But unfortunately we can show they do \emph{not} exist globally in general. So we use a trick, pulling back to a bigger space $\cA\to S^{[n_1]}\times S^{[n_2]}$ where there is a tautological such splitting.

\subsection{A splitting trick.} \label{splitpr}
Inside the total space of the bundle
$$
\cE\ :=\ (H^1)^*\!\otimes\!R^1\pi_*\;\cI_2\ \ \oplus\ \ (H^2)^*\!\otimes\!F_2
$$
over $S^{[n_1]}\times S^{[n_2]}$ there is a natural affine bundle\footnote{Modelled on the vector bundle $(H^1)^*\otimes\frac{\cO^{[n_2]}}{\cO}\ \ \oplus\,\ (H^2)^*\otimes\;\ker(h^2)$. Bhargav Bhatt pointed out that we could have used the Jouanolou trick here to find an affine bundle whose total space is an affine variety on which therefore there exist (non-canonical) splittings.} $\cA\subset\cE$ of pointwise splittings $(\phi_1,\phi_2)$ of \eqref{dig}. That is, the surjective map of locally free sheaves
$$
(1\otimes\iota\_2\,,1\otimes h^2)\,\colon\ \cE\To\ \End H^1\ \oplus\ \End H^2
$$
induces a map on the total spaces of the associated vector bundles. Taking the inverse image of the section  $\big(\!\id_{H^1},\id_{H^2}\!\big)$ defines the affine bundle
$$
\rho\ \colon\ \cA\To S^{[n_1]}\times S^{[n_2]}.
$$

Pulling \eqref{dig} back to $\cA$, it now has a \emph{canonical} tautological splitting $\Phi=(\phi_1,\phi_2)$, giving
\beq{donesplit}
\Phi\,:\ \rho^*H^1[-1]\ \oplus\ \rho^*H^2[-2]\To\rho^*R\pi_*\;\cI_2
\eeq
as sought in \eqref{split?}. That is, composing $\Phi$ with (the pullback by $\rho^*$ of) $\iota_2\colon R\pi_*\;\cI_2\to R^{\ge1}\pi_*\;\cO$ gives the identity: $\iota_2\circ\Phi=\id$.

So finally we may compose \eqref{donesplit} with (the pullback by $\rho^*$ of) $\iota_1^*$ \eqref{amic} to give a map
\beq{coney}
\iota_1^*\circ\Phi\,\colon\ \rho^*R^{\ge1}\pi_*\;\cO\To\rho^*R\hom_\pi(\cI_1,\cI_2).
\eeq
By construction, on taking $h^2$ it
induces (the pullback by $\rho^*$ of) the surjection \eqref{ont}. Therefore the cone $C(\iota_1^*\circ\Phi)$ on \eqref{coney} has no $h^2$ and is quasi-isomorphic to a 2-term complex of locally free sheaves.

We next give a more explicit description of $C(\iota_1^*\circ\Phi)$. It is nicest over $\rho^{-1}\big(S^{[n_1,n_2]}\big)$, since on $S^{[n_1,n_2]}$ the natural inclusion $\iota\colon\cI_1\to\cI_2$ induces a \emph{canonical} lift given by the composition
\beq{rite}
R^{\ge1}\pi_*\cO\To R\pi_*\cO\rt{\id}R\hom_\pi(\cI_1,\cI_1)\rt{\iota}R\hom_\pi(\cI_1,\cI_2).
\eeq

\begin{lem} \label{invariant}
The cone $C(\iota_1^*\circ\Phi)$ can be represented by a 3-term complex of vector bundles\footnote{This can be truncated to a 2-term complex of vector bundles by removing the third term and replacing the second term by the kernel of the surjection $(\rho^*\sigma_2,\psi_2)$.}
\beq{tri}
\xymatrix@R=0pt@C=35pt{
\rho^*E_0 \ar[r]^{\rho^*\sigma_1}& \rho^*E_1 \ar[r]^{\rho^*\sigma_2}&
\rho^*E_2, \\
\oplus & \oplus \\
\rho^*H^1 \ar[ruu]_{\psi_1} & \rho^*H^2 \ar[ruu]_{\psi_2}}
\eeq
where $E_0\to E_1\to E_2$ represents $R\hom_\pi(\cI_1,\cI_2)$.

Moreover the maps may be chosen so that, on restriction to $\rho^{-1}\big(S^{[n_1,n_2]}\big)$, they are the pullbacks by $\rho^*$ of maps on $S^{[n_1,n_2]}$, and $C(\iota_1^*\circ\Phi)$ is the pullback $\rho^*\;C$ of the cone $C$ on the composition \eqref{rite}.
\end{lem}

\begin{rmk} Recall that by our notation convention, we are using the same notation $\rho$ for the restriction of $\rho$ to $\rho^{-1}\big(S^{[n_1,n_2]}\big)$.

The Lemma tells us that on $\rho^{-1}\big(S^{[n_1,n_2]}\big)$, the explicit resolution \eqref{tri} can be taken to be constant on the fibres of $\rho$ --- i.e. independent on the choice of lifts $(\phi_1,\phi_2)$ of \eqref{dig} --- since, after composition with $\iota_1^*$, all lifts become quasi-isomorphic to the canonical one \eqref{rite} on $\rho^{-1}\big(S^{[n_1,n_2]}\big)$.
\end{rmk}

\begin{proof}
First we show that $C(\iota_1^*\circ\Phi)$ restricted to $\rho^{-1}\big(S^{[n_1,n_2]}\big)$ is quasi-isomorphic to $\rho^*\;C$.
Consider the diagram
$$
\xymatrix{
\rho^*R^{\ge1}\pi_*\cO \ar[r]\ar[d]_\Phi& \rho^*R\pi_*\cO \ar[r]^-{\id}& \rho^*R\hom_\pi(\cI_1,\cI_1) \ar[r]^{\iota}& \rho^*R\hom_\pi(\cI_1,\cI_2) \\
\rho^*R\pi_*\;\cI_2 \ar[ru]_(.6){\iota\_2}\ar[urrr]_{\iota_1^*}\ar@/^{-2ex}/[u]}
$$
on $\rho^{-1}\big(S^{[n_1,n_2]}\big)$, where we have the canonical map $\iota\colon\rho^*\cI_1\into\rho^*\cI_2$.
Here the curved arrow is from \eqref{split?} and makes the first triangle commute. Since by construction $\Phi$ is a right inverse to this map, the first triangle also commutes if we start at the top left corner. Since the second triangle also commutes, everything does, which means that $\iota_1^*\Phi$ equals the composition of the arrows along the top row.\medskip

Next we resolve $R\hom_\pi(\cI_1,\cI_2)^\vee$ by a complex of \emph{very negative} vector bundles $G\udot$. This means that they behave like projectives in the abelian category of coherent sheaves. In particular, by making them sufficiently negative, we can arrange that the map $(\iota_1^*\Phi)^\vee$ can be represented by a genuine map of complexes
\beq{firs}
\rho^*G\udot\To\rho^*(H^1)^*[1]\ \oplus\ \rho^*(H^2)^*[2],
\eeq
and, on $S^{[n_1,n_2]}$,
the dual of the composition
\eqref{rite} is represented by a genuine map of complexes
\beq{seco}
G\udot\To(H^1)^*[1]\,\oplus\,(H^2)^*[2].
\eeq
On restriction to $\rho^{-1}\big(S^{[n_1,n_2]}\big)\subset\cA$, we have shown that the first map \eqref{firs} is quasi-isomorphic to the pullback by $\rho^*$ of the second \eqref{seco}. Again we may assume we took the $G^i$ sufficiently negative that --- by the usual proof that quasi-isomorphic maps of complexes of projectives are homotopic --- there is a homotopy between \eqref{firs} and $\rho^*\;$\eqref{seco}. This homotopy is a pair of maps
$$
\rho^*G^0\To\rho^*(H^1)^*, \qquad \rho^*G^{-1}\To\rho^*(H^2)^*,
$$
over $\rho^{-1}\big(S^{[n_1,n_2]}\big)$. By the sufficient negativity of the $G^i$ they can be extended\footnote{For $N\gg0$ the restriction $\Hom\_{\cA}(G(-N),F)\to\Hom_{\rho^{-1}(S^{[n_1,n_2]})}(G(-N),F)$ is \emph{onto} for locally free $F$ and $G$.} to maps on all of $\cA$. Modifying \eqref{firs} by this homotopy, dualising and then truncating $(G\udot)^\vee$ to a 3-term complex now gives \eqref{tri}.
\end{proof}

So $C(\iota_1^*\circ\Phi)$ is quasi-isomorphic to the 2-term complex of vector bundles
\beq{EF}
\rho^*(E_0\oplus H^1)\rt\sigma F,
\eeq
where $F$ is defined to be the kernel
\beq{Fdef}
0\To F\To\rho^*(E_1\oplus H^2)\To\rho^*E_2\To0.
\eeq
And over $\rho^{-1}\big(S^{[n_1,n_2]}\big)$, the complex \eqref{EF} can be seen as a pull back by $\rho^*$.

\begin{lem}\label{gw}
The $h^0$ jumping locus of $C(\iota_1^*\circ\Phi)$ is the same as that of $\rho^*R\hom_\pi(\cI_1,\cI_2)$, i.e. it is $\rho^{-1}\big(S^{[n_1,n_2]}\big)$.
\end{lem}

\begin{proof}
Given any map $T\rt{f}\cA\to S^{[n_1]}\times S^{[n_2]}$, we denote the basechange of $\pi$ by
$$
\pi\_T\,\colon\, T\times S\To T.
$$
We denote the pull backs of $\cI_1,\,\cI_2$ to $T\times S$ by the same notation. Pulling $C(\iota_1^*\circ\Phi)$ back to $T$, the long exact sequence associated to the cone becomes
$$
0\To\hom_{\pi\_T}(\cI_1,\cI_2)\To h^0\big(f^*C(\iota_1^*\circ\Phi)\big)\To R^1\pi\_{T*}\cO\rt{\iota_1^*\Phi}\ext^1_{\pi\_T}(\cI_1,\cI_2).
$$
It remains to prove that the last arrow is an injection, since that implies $\hom_{\pi\_T}(\cI_1,\cI_2)\cong h^0\big(f^*C(\iota_1^*\circ\Phi)\big)$ on any $T$, to which we can apply Lemma \ref{h0base} to conclude.

The last arrow is the composition $\iota_1^*\circ\Phi$ in the diagram
$$\xymatrix{
R^1\pi\_{T*}\;\cO \ar@{.>}@/^{2ex}/[r]^\Phi& R^1\pi\_{T*}\;\cI_2 \ar[l]^{\iota\_2}\ar[d]^{\iota_1^*} \\
\ext^1_{\pi\_T}(\cI_1,\cO) \ar@<-1ex>@{<-^)}[u]+<0pt,-12pt>^-{\iota_1^*} & \ext^1_{\pi\_T}(\cI_1,\cI_2). \ar[l]_{\iota\_2}}
$$
To prove it is an injection it is sufficient to do so after composing with $\iota\_2$ along the bottom. Since the diagram commutes and $\Phi$ is a right inverse of the $\iota\_2$ along the top, this is equivalent to the left hand $\iota_1^*$ being injective. But this follows from the vanishing of $\ext^1_{\pi\_T}(\cO/\cI_1,\cO)$.
%We can express $\rho^*R\hom_\pi(\cI_1,\cI_2)$ as a 3-term complex of vector bundles $E_0\to E_1\to E_3$ and then $C(\iota_1^*\circ\Phi)$ as the 3-term complex
%$$
%\xymatrix@R=0pt{
%E_0 \ar[r]& E_1 \ar[r]& E_2. \\
%\oplus & \oplus \\
%H^1(\cO_S)\otimes\cO_{\cA} \ar[ruu]& H^2(\cO_S)\otimes\cO_{\cA} \ar[ruu]}
%$$
\end{proof}

For brevity we set $Z:=S^{[n_1,n_2]}$. By Lemmas \ref{gw} and \ref{invariant} we can see $\rho^{-1}(Z)$ as the degeneracy locus of any of the four maps
\begin{eqnarray}
\rho^*\sigma_1 \!&\colon& \rho^*E_0\To\rho^*E_1, \label{1} \\
(\rho^*\sigma_1,\psi_1) \!&\colon& \rho^*(E_0\oplus H^1)\To\rho^*E_1, \label{2} \\
\mat{\rho^*\sigma_1}{\psi_1}00 \!\!\!&\colon& \rho^*(E_0\oplus H^1)\To\rho^*(E_1\oplus H^2), \label{3} \\
\sigma \!&\colon& \rho^*(E_0\oplus H^1)\To K,
\label{4}
\end{eqnarray}
where $K:=\ker\!\big(\rho^*(E_1\oplus H^2)\to\rho^*E_2\big)$. These give rise to four different perfect obstruction theories for $\rho^{-1}(Z)$. The one we are interested in is the fourth \eqref{4}, but we will use the third \eqref{3} and the second \eqref{2} to relate this to the first \eqref{1} which has the desirable property that it is $\rho$-invariant: it is pulled back from a perfect obstruction theory on $Z$.
\medskip

By Lemma \ref{invariant} we can write each of (\ref{1}--\ref{4}) as 
the degeneracy locus of a map
$$
s\,\colon\ \rho^*A\To B,
$$
which on restriction to $\rho^{-1}(Z)$ becomes a pullback from $Z$ --- i.e. there exists a bundle $B'$ on $Z$ and $s'\colon A|_Z\to B'$ such that
\beq{pullbk}
B|_{\rho^{-1}(Z)}\cong\rho^*B'\quad\mathrm{and}\quad s|_{\rho^{-1}(Z)}\cong\rho^*s'.
\eeq
Now apply Section \ref{degloc} with $r_0-r=1$ to this. 
We see $\rho^{-1}(Z)$ as being cut out of
$$
\rho^*\PP(A)\ \cong\ \PP(\rho^*A)\rt{q}\cA
$$
by the induced section $\widetilde s$ \eqref{tilda} of $q^*B(1)$, inducing the perfect obstruction theory \eqref{mod3}
\beq{rho1}
\xymatrix@R=18pt@C=35pt{
q^*B^*(-1)\big|_{\rho^{-1}(Z)} \ar[d]_{\widetilde s}\ar[r]^-{d\;\widetilde s}& \Omega\_{\rho^*\PP(A)}\big|_{\rho^{-1}(Z)} \ar@{=}[d] \\
\rho^*(I/I^2) \ar[r]^-d& \Omega\_{\rho^*\PP(A)}\big|_{\rho^{-1}(Z)}\,.\!\!\!}
\eeq
Here $I$ is the ideal of $Z\subset \PP(A)$, so the bottom row is the truncated cotangent complex $\LL_{\rho^{-1}(Z)}$\;.

The bottom arrow factors through $\rho^*\Omega_{\PP(A)}|_{\rho^{-1}(Z)}$, so using \eqref{pullbk} the diagram factors through
\beq{doh}
\xymatrix@R=18pt{
q^*\rho^*(B')^*(-1)\big|_{\rho^{-1}(Z)} \ar[d]_{\widetilde s}\ar[r]^-{d\;\widetilde s}& \rho^*\Omega\_{\PP(A)}\big|_{\rho^{-1}(Z)} \ar@{=}[d] \\
\rho^*(I/I^2) \ar[r]^-d& \rho^*\Omega\_{\PP(A)}\big|_{\rho^{-1}(Z)}\,.\!\!\!}
\eeq
All of the sheaves here are pullbacks by $\rho^*$. Although on $\rho^{-1}(Z)$ the map $s$ is also a pullback \eqref{pullbk}, that does \emph{not} immediately mean that the maps in the above diagram are pulled back --- they use the restriction of $s$ not just to $\rho^{-1}(Z)$ but to its scheme theoretic doubling defined by the ideal $\rho^*I^2$.

However, in the first set-up \eqref{1} the maps clearly are pulled back. Using the second \eqref{2} and third \eqref{3} we will prove the same is true for the fourth \eqref{4}, so that it descends to give a perfect obstruction theory for $Z$ independent of the $(\phi_1,\phi_2)$ choices built into $\cA$.

\begin{prop} \label{rhocom} Using the description \eqref{4} of $\rho^{-1}(Z)$, the resulting diagram \eqref{doh} is $\rho$-invariant: it is the pullback by $\rho^*$ of a perfect obstruction theory $F\udot\to\LL_Z$ for $Z=S^{[n_1,n_2]}$.
\end{prop}

\begin{proof}
Applying \eqref{doh} to the first set-up \eqref{1} gives
$$
\xymatrix@R=18pt@C=40pt{
\rho^*q^*E_1^*(-1)\big|_{\rho^{-1}(Z)} \ar[d]_{\rho^*\widetilde\sigma_1}\ar[r]^-{\rho^*d(\widetilde\sigma_1)}& \rho^*\Omega\_{\PP(E_0)}\big|_{\rho^{-1}(Z)} \ar@{=}[d] \\
\rho^*(I/I^2) \ar[r]^-d& \rho^*\Omega\_{\PP(E_0)}\big|_{\rho^{-1}(Z)}\,,\!\!\!}
$$
where $I$ is the ideal of $Z\subset\PP(E_0)$.

Applied instead to the second \eqref{2}, we get the diagram 
\beq{dghell}
\xymatrix@R=18pt@C=50pt{
\rho^*q^*E_1^*(-1)\big|_{\rho^{-1}(Z)} \ar[d]_{\widetilde{(\rho^*\sigma_1,\psi_1)}}\ar[r]^-{d\widetilde{(\rho^*\sigma_1,\psi_1)}}& \rho^*\Omega\_{\PP(E_0\oplus H^1)}\big|_{\rho^{-1}(Z)} \ar@{=}[d] \\
J/J^2 \ar[r]^-d& \rho^*\Omega\_{\PP(E_0\oplus H^1)}\big|_{\rho^{-1}(Z)}\,,\!\!\!}
\eeq
where $J$ is the ideal of $\rho^{-1}(Z)\subset\PP(\rho^*E_0\oplus H^1)$. (Throughout this proof we denote $q^*H^i,\,\rho^*H^i$ and $q^*\rho^*H^i$ simply by $H^i$.) This inclusion factors
$$
\rho^{-1}(Z)\,\subset\,\PP(\rho^*E_0)\,\subset\,\PP(\rho^*E_0\oplus H^1).
$$
The first has conormal sheaf $\rho^*I/I^2$, while the second has conormal bundle $(H^1)^*(-1)$
%, as follows. Pushing $\cO(1)$ down from $\PP(\rho^*E_0\oplus H^1)$ to $\cA$ gives
%$$
%q_*\;\cO(1)\=(\rho^*E_0\oplus H^1)^*.
%$$
%Taking the $(H^1)^*$ summand, pulling back and multiplying by sections of $\cO(-1)$ defines the linear functions on the fibres of $\PP(\rho^*E_0\oplus H^1)$ which vanish on $\PP(\rho^*E_0)$ and generate its ideal. Therefore\footnote{Perhaps an easier way to see this is to recognise $q^*(H^1)^*(-1)$ as the conormal bundle to $\PP(\rho^*E_0)\subset\PP(\rho^*E_0\oplus H^1)$.}
The splitting of $\rho^*E_0\oplus H^1$ induces a splitting
$$
\Omega_{\PP(\rho^*E_0\oplus H^1)}|_{\rho^{-1}(Z)}\cong\Omega_{\PP(\rho^*E_0)}|_{\rho^{-1}(Z)}\oplus H(-1)|_{\rho^{-1}(Z)}
$$
and so
$$
J/J^2\=\rho^*(I/I^2)\oplus(H^1)^*(-1).
$$
When substituted into \eqref{dghell} it becomes
\beq{arss}
\xymatrix@R=18pt@C=40pt{
\rho^*q^*E_1^*(-1)\big|_{\rho^{-1}(Z)} \ar[d]_{(\rho^*\widetilde\sigma_1,\,\psi_1^*)}\ar[r]^-{(\rho^*d\;\widetilde\sigma_1,\,\psi_1^*)}& \rho^*\Omega\_{\PP(E_0)}\big|_{\rho^{-1}(Z)} \oplus(H^1)^*(-1)\ar@{=}[d] \\
\rho^*(I/I^2)\oplus(H^1)^*(-1) \ar[r]^-{(d,\id)}& \rho^*\Omega\_{\PP(E_0)}\big|_{\rho^{-1}(Z)}\oplus(H^1)^*(-1)\,.\!\!\!}
\eeq
The \emph{key point} of this proof is that the above diagram is pulled back by $\rho^*$ from a similar diagram on $Z$. This is clear of all the bundles involved, and also clear of the first summand of the upper and left hand arrows. But these are the only parts of the arrows which depend on the thickening of $\rho^{-1}(Z)$. The other summands $\psi_1^*$ depend only on their restriction to $\rho^{-1}(Z)$, where they are also pull backs by Lemma \ref{invariant}.

So the second degeneracy locus description of $\rho^{-1}(Z)$ \eqref{2} gives rise to a diagram which descends to (a perfect obstruction theory on) $Z$.
For the third description \eqref{3} we add an extra $(H^2)^*(-1)$ summand to the diagram \eqref{arss} with all maps from it zero:
$$\qquad
\xymatrix@R=20pt@C=60pt{
\rho^*q^*(E_1\oplus H^2)^*(-1)\big|_{\rho^{-1}(Z)} \ar[d]_{(\rho^*\widetilde\sigma_1,\,\psi_1^*)}^{\!\oplus\,(0,0)}\ar[r]^-{(\rho^*d\;\widetilde\sigma_1,\,\psi_1^*)\!}_-{\oplus\,(0,0)}& \rho^*\Omega\_{\PP(E_0)}\big|_{\rho^{-1}(Z)} \oplus(H^1)^*(-1)\ar@{=}[d] \\
\rho^*(I/I^2)\oplus(H^1)^*(-1) \ar[r]^-{(d,\id)}& \rho^*\Omega\_{\PP(E_0)}\big|_{\rho^{-1}(Z)}\oplus(H^1)^*(-1).\!\!} \vspace{-17mm}
$$
\beq{bb}
\vspace{1cm}
\eeq
This is therefore also a pullback by $\rho^*$. Finally, since \eqref{tri} is a complex, the map \eqref{3} takes values in $K\subset\rho^*(E_1\oplus H^2)$. Thus the equation cutting out $\rho^{-1}(Z)$ takes values in $q^*K(1)\subset q^*\rho^*(E_1\oplus H^2)(1)$. Therefore the 
upper horizontal and left hand vertical arrows of \eqref{bb} factor through $q^*K^*(-1)$, giving
$$
\xymatrix@R=18pt@C=40pt{
q^*K^*(-1)\big|_{\rho^{-1}(Z)} \ar[d]\ar[r]& \rho^*\Omega\_{\PP(E_0)}\big|_{\rho^{-1}(Z)} \oplus(H^1)^*(-1)\ar@{=}[d] \\
\rho^*(I/I^2)\oplus(H^1)^*(-1) \ar[r]^-{(d,\id)}& \rho^*\Omega\_{\PP(E_0)}\big|_{\rho^{-1}(Z)}\oplus(H^1)^*(-1)\,,\!\!\!} \vspace{-17mm}
$$
\beq{rho2}
\vspace{8mm}
\eeq
which is the diagram \eqref{doh} applied to the fourth degeneracy locus \eqref{4}. 

By Lemma \ref{invariant}, both $K$ and its inclusion into $\rho^*E_1\oplus H^2$ are $\rho$-invariant. Thus the quotient diagram \eqref{rho2} of the diagram \eqref{bb} is also a pull back by $\rho^*$.
\end{proof}

%\begin{rmk}
%We have induced a perfect obstruction theory on $Z$ by descending one down the fibres of $\rho$.
%Another way is to use local sections of $\rho$. Writing these locally as complete intersections, we add more equations to those of $\widetilde\sigma$ cutting out $\rho^{-1}(Z)$. It is easy to show that the resulting perfect obstruction theory on the local section, thought of as a local copy of $Z$, is independent of the choice of section.
%\end{rmk}

\begin{proof}[Proof of Theorem \ref{aim}] Applying \eqref{rho1} (with $A=E_0\oplus H^1$ and $B=K$) to the fourth description \eqref{4} induces a perfect obstruction theory on $\rho^{-1}\big(S^{[n_1,n_2]}\big)$. And diagram \eqref{doh} applied to \eqref{4} gives \eqref{rho2}, which descends --- by Proposition \ref{rhocom} --- to give a compatible perfect obstruction theory on $S^{[n_1,n_2]}$. This compatibility means they satisfy
$$
\rho^*\big[S^{[n_1,n_2]}\big]^{\vir}\=
\big[\rho^{-1}\big(S^{[n_1,n_2]}\big)\big]^{\vir}
\ \in\ A_{\;\dim\cA-k}(\cA).
$$
By Theorem \ref{prop} the second term is $\Delta\;_{r_1-r}^{r_0-r}\big(c(K-(\rho^*E_0\oplus H^1))\big)
$. But the Chern classes of $K-(\rho^*E_0\oplus H^1)$ are the same as those of $\rho^*(-E_0+E_1-E_2)$ and so those of $\rho^*R\hom_\pi(\cI_1,\cI_2)[1]$. Thus
$$
\rho^*\big[S^{[n_1,n_2]}\big]^{\vir}\=\rho^*
\Delta\;_{r_1-r}^{r_0-r}\big(c(R\hom_\pi(\cI_1,\cI_2)[1])\big)\ \in\ A_{\;\dim\cA-k}(\cA).
$$
Here $r_0-r=1$ is the rank of $\ker(\rho^*E_0\to\rho^*E_1)$ over the degeneracy locus. And $r_1-r_0=\rk K-\rk E_0-h^1(\cO_S)=\rk E_1+h^2(\cO_S)-\rk E_2-\rk E_0-h^1(\cO_S)=-\chi(I_1,I_2)+\chi(\cO_S)-1=n_1+n_2-1$, so $r_1-r=n_1+n_2$ and $k=(r_0-r)(r_1-r)=n_1+n_2$. Therefore the above becomes
$$
\rho^*\big[S^{[n_1,n_2]}\big]^{\vir}\=\rho^*\,
c\_{n_1+n_2}(R\hom_\pi(\cI_1,\cI_2)[1])\ \in\ A_{\;\dim\cA-n_1-n_2}(\cA).
$$
But since $\rho$ is an affine bundle,
\beq{krush}
\rho^*\,\colon\,A_{n_1+n_2}(S^{[n_1]}\times S^{[n_2]})\To A_{\;\dim\cA-n_1-n_2}(\cA)
\eeq
is an isomorphism \cite[Corollary 2.5.7]{Kr},
%\cite[Lemma 2.2]{To},
so the result follows.
\end{proof}

Over the degeneracy locus $\rho^{-1}\big(S^{[n_1,n_2]}\big)$, our complex $C(\iota_1^*\Phi)$ has
$$
h^0\=\cO,
$$
trivialised by the inclusion $\iota:\cI_1\into\cI_2$. And $h^1[-1]$ is the cone on
$$
h^0\big(C(\iota_1^*\Phi)\big)\ \cong\ \cO_{\rho^{-1\,}S^{[n_1,n_2]}}\,\rt{h^0}\,C(\iota_1^*\Phi)\big|_{\rho^{-1\,}S^{[n_1,n_2]}}\,.
$$
By Lemma \ref{invariant} and the description \eqref{rite}, this is
\beq{trayce}
R\hom_p(\cI_1,\cI_2)\_0\ :=\ \Cone\!\;\big(Rp_*\;\cO\rt{\iota\cdot\id}R\hom_p(\cI_1,\cI_2)\big),
\eeq
where we recall that $p$ is the basechange of $\pi$ to $S^{[n_1,n_2]}\subset S^{[n_1]}\times S^{[n_2]}$. Thus
\beq{h11}
h^1\=\ext^1_p(\cI_1,\cI_2)\_0.
\eeq
Theorem \ref{prop} shows the
perfect obstruction theory of a degeneracy locus has virtual tangent bundle
$$
T_{\cA}|_{\rho^{-1}(Z)}\To(h^0)^*\otimes h^1.
$$
As in the proof of Theorem \ref{aim} this descends to give our perfect obstruction theory on $Z=S^{[n_1,n_2]}$, yielding the following.

\begin{cor}\label{vtb}
The perfect obstruction theory on $S^{[n_1,n_2]}$ of Theorem \ref{aim} can be written, in the notation of \eqref{trayce}, as
\beq{tun}
\big\{T_{S^{[n_1]}\times S^{[n_2]}}\big|_{S^{[n_1,n_2]}}\to\ext^1_p(\cI_1,\cI_2)\_0\big\}^\vee\To\LL_{S^{[n_1,n_2]}}.
\eeq
\end{cor}

\section{$k$-step nested Hilbert schemes} \label{kstep}
For $n_1\ge n_2\ge\cdots\ge n_k$, the $k$-step Hilbert scheme
$$
S^{[n_1,n_2,\ldots,n_k]}\ :=\ \big\{I_1\subseteq I_2\subseteq\cdots\subseteq I_k\subseteq\cO_S,\ \mathrm{length}(\cO_S/I_i)=n_i\big\}
$$
can be seen inside $S^{[n_1]}\times\cdots\times S^{[n_k]}$ as the intersection of the $(k-1)$ degeneracy loci
$$
\big\{\!\Hom(I_i,I_{i+1})=\C\big\},\quad i=1,2,\ldots,k-1
$$
where the maps in the complexes $R\hom_\pi(\cI_i,\cI_{i+1})$ drop rank.

So when $H^{\ge1}(\cO_S)=0$ we can employ the exact same method as in Proposition \eqref{nest0}, using $k-1$ sections of tautological bundles on a $(k-1)$-fold fibre product of relative Grassmannians, to describe a perfect obstruction theory, virtual cycle, and product of Thom-Porteous terms to compute its pushforward.

For general $S$, possibly with $H^{\ge1}(\cO_S)\ne0$, we can replace the complexes $R\hom_\pi(\cI_i,\cI_{i+1})$ with their modifications $C(\iota_i^*\circ\Phi_i)$ of \eqref{coney} after pulling back to an affine bundle of splittings. Then we use the same method as in Theorem \ref{aim} to produce the following result. 
We use the projections
\begin{align*}
\pi\,\colon\ S^{[n_1]}\times\cdots\times S^{[n_k]}\times S\ &\To\, S^{[n_1]}\times\cdots\times S^{[n_k]}, \\ p\,\colon\ S^{[n_1,\ldots,n_k]}\times S\ &\To\,S^{[n_1,\ldots,n_k]},
\end{align*}
and, when $I\subset J$, the same $\Ext(I,J)\_0$ notation as in (\ref{trayce}, \ref{h11}).

\begin{thm} \label{nest}
Fix a smooth complex projective surface $S$. Via degeneracy loci the $k$-step nested Hilbert scheme $S^{[n_1,\ldots,n_k]}$ inherits a perfect obstruction theory $F\udot\to\LL_{S^{[n_1,\ldots,n_k]}}$ with virtual tangent bundle
$$
(F\udot)^\vee\ \cong\ \Big\{T_{S^{[n_1]}}\oplus\cdots\oplus T_{S^{[n_k]}}\To\ext^1_p(\cI_1,\cI_2)\_0\oplus\cdots\oplus\ext^1_p(\cI_{k-1},\cI_k)\_0\Big\},
$$
where the arrow is the obvious direct sum of the maps \eqref{tun}. This is isomorphic to the virtual tangent bundle
$$
\Cone\left\{\bigg(\bigoplus_{i=1}^kR\hom_p(\cI_i,\cI_i)\bigg)_{\!0}
\To\bigoplus_{i=1}^{k-1}R\hom_p(\cI_i,\cI_{i+1})\right\}
$$
of the perfect obstruction theory of \cite{GSY1} or Vafa-Witten theory \cite{TT1} when the latter are defined. The pushforward of the resulting virtual cycle
$$
\big[S^{[n_1,\ldots,n_k]}\big]^{\vir}\,\in\,A_{n_1+n_k}\big(S^{[n_1,\ldots,n_k]}\big)
$$
to $S^{[n_1]}\times\cdots\times S^{[n_k]}$ is given by the product
$$
c_{n_1+n_2}\big(R\hom_{\pi}(\cI_1,\cI_2)[1]\big)\cup\cdots\cup c_{n_{k-1}+n_k}\big(R\hom_{\pi}(\cI_{k-1},\cI_k)[1]\big).
$$
\end{thm}

\begin{rmk} Note that we are not claiming the two perfect obstruction theories are the same, although they undoubtedly are. Proving this would involve identifying the map $F\udot\to\LL$ produced by our degeneracy locus construction with the one induced by Atiyah classes in \cite{GSY2,TT1}. We do not need this because the virtual cycles depend only on the scheme structure of $S^{[n_1,\ldots, n_k]}$ and the K-theory class of $F\udot$.
\end{rmk} \vspace{-3mm}

\begin{proof} All that is left to do is relate the two virtual tangent bundles. The virtual tangent bundle of \cite{GSY1} is the cone on the bottom row of the diagram 
\beq{gsy}
\xymatrix@R=18pt{
Rp_*\;\cO \ar[d]^(.42){\oplus_{i=1}^k\id} \\
\bigoplus_{i=1}^kR\hom_p(\cI_i,\cI_i) \ar[r]\ar[d]& \bigoplus_{i=1}^{k-1}R\hom_p(\cI_i,\cI_{i+1}) \ar@{=}[d] \\
\bigg(\!\bigoplus_{i=1}^kR\hom_p(\cI_i,\cI_i)\bigg)_{\!0} \ar[r]& \bigoplus_{i=1}^{k-1}R\hom_p(\cI_i,\cI_{i+1}).\!\!}
\eeq
Here the left hand column is an exact triangle which defines the term in the lower left corner.
The central horizontal arrow acts on the $j$th summand ($1\le j\le k$) of the left hand side by taking it to $(0,\ldots,0,-i_{j-1}^*,i_j,0,\ldots,0)$ on the right hand side, where $i_j$ appears in the $j$th position and is the canonical map $\cI_j\into \cI_{j+1}$. (For $j=1$ we ignore the $-i_{j-1}^*$ term to get $(i_1,0,\ldots,0)$; for $j=k$ we ignore the $i_j$ term to get $(0,\ldots,0,-i^*_{k-1})$.) This has zero composition with
$\oplus_{i=1}^k\id$, so induces the lower horizontal arrow.

The identity map from $(Rp_*\;\cO)^{\oplus k}=Rp_*\;\cO\otimes\C^k$ to the central left hand term of \eqref{gsy} induces a map from $Rp_*\;\cO\otimes(\C^k/\C)$ to the bottom left hand term, where $\C$ sits in $\C^k$ via $(1,1,\ldots,1)$. Projecting the elements $(1,0,\ldots,0)$, $(1,1,0,\ldots,0),\,\ldots,\,(1,1,\ldots,1,0)$ of $\C^k$ defines a basis in $\C^k/\C$ and so identifies $Rp_*\;\cO\otimes(\C^k/\C)\cong(Rp_*\;\cO)^{\oplus(k-1)}$. Using our description of the central arrow, this identifies the induced map
$$
Rp_*\;\cO\otimes(\C^k/\C)\To\bigoplus_{i=1}^{k-1}R\hom_p(\cI_i,\cI_{i+1})
$$
with
$$
(Rp_*\;\cO)^{\oplus(k-1)}\rt{\operatorname{diag}(i_1,i_2,\cdots,i_{k-1})}\bigoplus_{i=1}^{k-1}R\hom_p(\cI_i,\cI_{i+1}).
$$
Taking the cone on these two maps from $(Rp_*\;\cO)^{\oplus(k-1)}$ to the two entries on the bottom row of \eqref{gsy} shows the bottom row is quasi-isomorphic to
$$
\bigoplus_{i=1}^kR\hom_p(\cI_i,\cI_i)\_0\To\ 
\bigoplus_{i=1}^{k-1}R\hom_p(\cI_i,\cI_{i+1})\_0
$$
in the notation of \eqref{trayce}.
Each of these complexes has cohomology only in degree 1, so the virtual tangent bundle of \cite{GSY1} is the cone on
$$
\bigoplus_{i=1}^k\ext^1_p(\cI_i,\cI_i)\_0\To\ 
\bigoplus_{i=1}^{k-1}\ext^1_p(\cI_i,\cI_{i+1})\_0
$$
in the notation of \eqref{h11}. On the $j$th summand on the left the arrow is $(0,\ldots,0,-i_{j-1}^*,i_j,0,\ldots,0)$. But this is $(F\udot)^\vee$, as required.

In \cite{GSY2} it is shown that the perfect obstruction theory of \cite{GSY1} is a summand of the obstruction theory one gets from localised local DT theory. The piece one has to remove is explained in terms of a more global perfect obstruction theory arising in Vafa-Witten theory in \cite{TT1}.
\end{proof}

\section{Generalised Carlsson-Okounkov vanishing} \label{COvan}

Theorem \ref{aim} expresses $\big[S^{[n_1,n_2]}\big]^{\vir}$ as a degeneracy class. This allows us to give a topological proof of the following result of Carlsson-Okounkov \cite{CO}, which we will then generalise below.

\begin{cor}
Let $S$ be any smooth projective surface. Over $S^{[n_1]}\times S^{[n_2]}$ we have the vanishing
\beq{COzero}
c_{n_1+n_2+i}\big(R\hom_{\pi}(\cI_1,\cI_2)[1]\big)\=0, \qquad i>0.
\eeq
\end{cor}

\begin{proof} 
We apply the higher Thom-Porteous formula \eqref{higher} to our modified complex $C(\iota_1^*\circ\Phi)$ \eqref{coney} on $\cA$. It has degeneracy locus $\rho^{-1}\big(S^{[n_1,n_2]}\big)$, over which $h^0$ is just $\cO$, trivialised by the tautological inclusion $\cI_1\into\cI_2$ over the nested Hilbert scheme. Hence \eqref{higher} gives
$$
c_{r_1-r_0+i+1}\big(C(\iota_1^*\circ\Phi)[1]\big)\=0
$$
for $i>0$, where $r_1-r_0=n_1+n_2-1$.

Since $C(\iota_1^*\circ\Phi)[1]$ only differs from $\rho^*R\hom_{\pi}(\cI_1,\cI_2)[1]$ by some trivial bundles $H^1,\,H^2$, this gives
$$
\rho^*\;c_{n_1+n_2+i}\big(R\hom_{\pi}(\cI_1,\cI_2)[1]\big)\=0.
$$
But $\rho^*\colon A_{n_1+n_2-i}(S^{[n_1]}\times S^{[n_2]})\to A_{\;\dim\cA-n_1-n_2-i}(\cA)$
is an isomorphism \cite[Corollary 2.5.7]{Kr}, which gives the result.
\end{proof}

The rest of this Section is devoted to proving the following generalisation.

\begin{thm} \label{park}
Let $S$ be any smooth projective surface. For any curve class $\beta\in H_2(S,\Z)$, any Poincar\'e line bundle $\cL\to S\times\Pic_{\beta}(S)$, and any $i>0$,
\beq{film}
c_{n_1+n_2+i}\big(\!\;R\pi_*\;\cL-R\hom_{\pi}(\cI_1,\cI_2\otimes\cL)\big)\=0
\eeq
on $S^{[n_1]}\times S^{[n_2]}\times\Pic_\beta(S)$.
\end{thm}

To prove this we will work with more general nested Hilbert schemes of subschemes $S\supset Z_1\supseteq Z_2$, by allowing $Z_1$ to have dimension $\le1$ instead of just 0. Separating out its divisorial and 0-dimensional parts, we are then led, for $\beta\in H_2(S,\Z)$, to the nested Hilbert scheme $S^{[n_1,n_2]}_{\beta}$. As a set it is
\begin{multline} \label{nestbeta}
S^{[n_1,n_2]}_{\beta}\ :=\ \big\{I_1(-D)\subset I_2\subset\cO_S\ \colon \\\mathrm{length}(\cO_S/I_i)=n_i,\ D\ \mathrm{Cartier\ with}\ [D]=\beta\big\}.
\end{multline}
As a scheme it represents the functor taking schemes $B$ to families of nested ideals $\cI_1(-\cD)\into \cI_2\into\cO_{S\times B}$, flat over $B$. Here $\cD$ is a Cartier divisor, the $\cO_S/\cI_i$ are finite over $B$ of length $n_i$, and --- on restriction to any closed fibre $S_b$ --- $\cD_b$ has class $\beta$ and the maps are still injections.

Setting $\beta=0$ and $n_1\ge n_2$ recovers the punctual nested Hilbert scheme \eqref{nst}. Instead setting $n_1=0=n_2$ gives the Hilbert scheme of curves $S_\beta$, which fibres over $\Pic_\beta(S)\ni L$ with fibres $\PP(H^0(L))$. \medskip

In the sequel \cite{GT2} we will construct a natural perfect obstruction theory and virtual cycle on $S^{[n_1,n_2]}_{\beta}$ for any $\beta$. Here we only sketch a less general construction for classes $\beta\gg0$ since we do not actually need the virtual class, only the degeneracy locus expression, in order to prove Theorem \ref{park}.

\subsection{Another degeneracy locus construction}
So fix $\beta\gg0$ sufficiently positive that $H^{\ge1}(L)=0$ for all $L\in\Pic_\beta(S)$. The Abel-Jacobi map $\AJ\colon S_\beta\to\Pic_\beta(S)$ is then a projective bundle. 
Let $\cD$ be the universal curve in $S_\beta\times S$ (or any basechange thereof) and as usual let $\pi$ denote any projection down $S$. Then
$$
R\hom_\pi(\cI_1(-\cD),\cO) \quad\mathrm{over}\quad S^{[n_1]}\times S^{[n_2]}\times S_\beta
$$
has $h^2=0$. Also $h^0=\pi_*\;\cO(\cD)$ and
$$
h^1\=\ext^1_\pi(\cI_1(-\cD),\cO)\ \cong\ \ext^2_\pi(\cO_{\cZ_1}(-\cD),\cO)\ \cong\ \Big[\big(K_S(-\cD)\big)^{[n_1]}\Big]^*,
$$
with the last isomorphism\footnote{Given any line bundle $L$ on $S$, there is a tautological rank $n_1$ vector bundle $L^{[n_1]}:=\pi_*\big[(\cO_{S^{[n_1]}}\boxtimes L)\otimes\cO_{\cZ_1}\big]$ over $S^{[n_1]}$ whose fibre over $Z_1\in S^{[n_1]}$ is $\Gamma(L|_{Z_1})$. Here we are using the obvious family generalisation applied to the line bundle $K_S(-\cD)$ over $S\times S_\beta$.} given by Serre duality down the fibres of $\pi$.

Thus $R\hom_\pi(\cI_1(-\cD),\cO)$ can be trimmed to a 2-term complex of vector bundles $E_0\to E_1$ sitting in an exact sequence
$$
0\To\pi_*\;\cO(\cD)\To E_0\To E_1\To\Big[\big(K_S(-\cD)\big)^{[n_1]}\Big]^*\To0,
$$
all of whose terms are locally free.

So just as in Section \ref{splitpr} we may work on an affine bundle $\rho\colon\cA\to S^{[n_1]}\times S^{[n_2]}\times S_\beta$ over which this splits canonically, giving an isomorphism
$$
\rho^*R\hom_\pi(\cI_1(-\cD),\cO)\ \cong\ \rho^*\pi_*\;\cO(\cD)\ \oplus\ \rho^*\Big[\big(K_S(-\cD)\big)^{[n_1]}\Big]^*[-1]
$$
which induces the identity on cohomology sheaves.
From now on we shall omit $\rho^*$ from our notation and work as if this splitting holds on $S^{[n_1]}\times S^{[n_2]}\times S_\beta$ since we know that $\rho^*$ induces an isomorphism on Chow groups \eqref{krush}.

In particular we get an induced composition
$$\xymatrix{
R\hom_\pi(\cI_1(-\cD),\cI_2) \ar[r]\ar[rrd]_{\Psi}& R\hom_\pi(\cI_1(-\cD),\cO) \ar[r]& \pi_*\;\cO(\cD) \ar[d] \\ && \frac{\pi_*\;\cO(\cD)}{s_{\cD}\cdot\;\cO}\;,\!}
\vspace{-14mm}
$$
\beq{sie}\vspace{6mm}\eeq
where $s_{\cD}\colon\cO\to\pi_*\;\cO(\cD)$ is induced by adunction from the section $s_{\cD}\colon\pi^*\cO\to\cO(\cD)$  cutting out $\cD$. At a closed point $(I_1,I_2,D)$ of $S^{[n_1]}\times S^{[n_2]}\times S_\beta$, the horizontal composition along the top of \eqref{sie} acts on $h^0$ as follows. It takes a nonzero element of $\Hom(I_1(-D),I_2)$ --- i.e. a point of the nested Hilbert scheme up to scale --- to its divisorial part in $H^0(\cO(D))$; this is injective. The vertical map then compares this to the divisor $D$. Thus $h^0(\Psi)$ has one dimensional kernel $\cO$ (canonically trivialised by $s_{\cD}$) at precisely the points of the nested Hilbert scheme
\beq{sett}
S^{[n_1,n_2]}_\beta\ \stackrel\iota\Into\ S^{[n_1]}\times S^{[n_2]}\times S_\beta,
\eeq
and the kernel is never any bigger. Said differently, the 2-term complex of vector bundles
$$
\Cone(\Psi)[-1]
$$
drops rank by $1$ on the subset \eqref{sett}, and no further. By working very similar to that in Proposition \ref{nest0} one can easily show that \eqref{sett} also describes the degeneracy locus scheme-theoretically, inducing a perfect obstruction theory on $S^{[n_1,n_2]}_\beta$. By the Thom-Porteous formula of Proposition \ref{simplah} the resulting virtual cycle therefore satisfies
$$
\iota_*\big[S^{[n_1,n_2]}_\beta\big]^{\vir}\=
c\_b\big(\!\Cone(\Psi)\big),
$$
where $b=\chi(\Cone(\Psi))+1=n_1+n_2$. More generally, by \eqref{higher},
$$
\iota_*\!\left(\!\;c_1\big((h^0)^*\big)^i\cap\big[S^{[n_1,n_2]}_\beta\big]^{\vir}\right)\=c_{n_1+n_2+i\;}\big(\!\Cone(\Psi)\big).
$$
Since we have already observed that $h^0(\Cone(\Psi)[-1])\cong\cO$ is trivialised by the restriction of $s_{\cD}$ to \eqref{sett}, this gives
\beq{park2}
c_{n_1+n_2+i}\big(R\pi_*\;\cO(\cD)-R\hom_\pi(\cI_1(-\cD),\cI_2)\big)\=0 \quad\text{on}\ S^{[n_1]}\times S^{[n_2]}\times S_\beta
\eeq
for $\beta\gg0$ and all $i>0$. Notice how close this is to the result claimed in Theorem \ref{park}. 
%We need to descend \eqref{park2} from $S_\beta$ to $\Pic_\beta(S)$, and generalise from $\beta\gg0$ to all $\beta\in H_2(S,\Z)$. \medskip
%
%Let $p\colon\Pic_\beta(S)\times S\to\Pic_\beta(S)$ denote the projection, and fix any Poincar\'e line bundle $\cL\to\Pic_\beta(S)\times S$ so that
%\beq{AJ}
%S_\beta\=\PP(p_*\;\cL)\rt{\AJ}\Pic_\beta(S).
%\eeq
%The tautological subbundle
%$$
%\cO_{\PP(p_*\;\cL)}(-1)\Into\AJ^*(p_*\;\cL)
%$$
%composed with (the pullback by $(\AJ\times1)^*$ of) $p^*p_*\;\cL\to\cL$ defines a section of $(\AJ\times1)^*\cL(1)$ cutting out the universal divisor $\cD\subset S_\beta\times S$. Therefore, suppressing the obvious pullback map for clarity, we can identify
%$$
%\cO(\cD)\ \cong\ \cL(1).
%$$

\begin{proof}[Proof of Theorem \ref{park}.]
We want to descend \eqref{park2} from $S_\beta$ to $\Pic_\beta(S)$ and then extend from $\beta\gg0$ to all $\beta\in H_2(S,\Z)$. We will use the formula of \cite[Proposition 1]{Ma},
$$
c_{n+i}(F\otimes M)\=\sum_{j=0}^{n+i}{\rk F-j\choose n+i-j}c_j(F)\;c_1(M)^{n+i-j},
$$
for any perfect complex $F$ and line bundle $M$, using the usual conventions for negative binomial coefficients. Applying this to $F=R\pi_*\;\cO(\cD)-R\hom_\pi(\cI_1(-\cD),\cI_2)$ of rank $n:=n_1+n_2$ gives
\beq{FoM}
c_{n_1+n_2+i}(F\otimes M)\=\sum_{j=n_1+n_2+1}^{n_1+n_2+i}{n_1+n_2-j\choose n_1+n_2+i-j}c_j(F)\;c_1(M)^{n_1+n_2+i-j},
\eeq
because for smaller $j$ the inequalities $n_1+n_2+i-j>n_1+n_2-j\ge0$ force the binomial coefficient to vanish. By the vanishing \eqref{park2} this gives
\beq{M}
c_{n_1+n_2+i}(F\otimes M)\=0
\eeq
for $i>0$ and any line bundle $M$ on $S^{[n_1]}\times S^{[n_2]}\times S_\beta$. For any Poincar\'e line bundle $\cL$ pulled back from $S\times\Pic_\beta(S)$, the line bundle $\cL(-\cD)$ is trivial on each $S$ fibre and is the pullback $\pi^*M$ of a line bundle $M$ on 
$S^{[n_1]}\times S^{[n_2]}\times S_\beta$. (In fact $M=\cO(-1)$ is the tautological bundle if we consider $S_\beta\to\Pic_\beta(S)$ to be the projectivisation of the vector bundle $\pi_*\;\cL$.) Subsituting into \eqref{M} gives
$$
c_{n_1+n_2+i}\big(R\pi_*\;\cL-R\hom_\pi(\cI_1,\cI_2\otimes\cL)\big)\=0
$$
on $S^{[n_1]}\times S^{[n_2]}\times S_\beta$. Since this is pulled back from $S^{[n_1]}\times S^{[n_2]}\times\Pic_\beta(S)$ the Leray-Hirsch theorem shows we have the same vanishing there.
\medskip

So we have proved the vanishing \eqref{film} for $\beta\gg0$, and we need to generalise it to all $\beta\in H_2(S,\Z)$.
%First we use a trick from \cite[Lecture 19]{Mu}. We assume our $\beta\gg0$ is positive enough that $\chi(L)>h^{0,1}(S)=\dim\Pic_\beta(S)$ when $c_1(L)=\beta$. Letting $p\colon\Pic_\beta(S)\times S\to\Pic_\beta(S)$ denote the projection, we have
%\beq{AJ}
%S_\beta\=\PP(p_*\;\cL)\rt{AJ}\Pic_\beta(S),
%\eeq
%for any Poincar\'e line bundle $\cL\to\Pic_\beta(S)\times S$. We will take sections of the Abel-Jacobi fibration \eqref{AJ}. These correspond to line subbundles (not subsheaves!)
%$$
%\cM\Into p_*\;\cL \quad\mathrm{on}\ \Pic_\beta(S).
%$$
%Equivalently, they correspond to nowhere zero sections $$
%\cO_{\Pic_\beta(S)}\To p_*\big(\cL\otimes p^*\cM^{-1}).
%$$
%Since by assumption $\rk p_*\;\cL>\dim\Pic_\beta(S)$, such sections exist as soon as the line bundle $\cM^{-1}$ is sufficiently positive.
%
%So, replacing $\cL$ by $\cL\otimes p^*\cM^{-1}$ for $\cM^{-1}\gg0$ if necessary, we may assume that $p_*\;\cL$ has a section $s$ without zeros. Therefore $\PP(\im s)$ is a section of the Abel-Jacobi fibration \eqref{AJ}. Taking the product with $S$, the section $s$ defines a section of $\cL\to\Pic_\beta(S)\times S$ with zeros
%$$
%\cD\ \cap\ \big(\!\Pic_\beta(S)\times S)\ \subset\ S_\beta\times S.
%$$
%Thus we can identify $\cL$ with the restriction of $\cO(\cD)$ to the section. Pulling back \eqref{park2} by this section then gives
%$$
%c_{n_1+n_2+i}\big(R\pi_*\;\cL-R\hom_\pi(\cI_1,\cI_2\otimes\cL)\big)\=0 \quad\text{on}\ S^{[n_1]}\times S^{[n_2]}\times\Pic_\beta(S).
%$$
%To extend this vanishing to all $\beta$ and $\cL$, we use
We write the left hand side of \eqref{film} on $S^{[n_1]}\times S^{[n_2]}\times\Pic_\beta(S)$ in terms of characteristic classes using the Grothendieck-Riemann-Roch theorem applied to $\pi$. The result is an $H^{2(n_1+n_2+i)}\big(S^{[n_1]}\times S^{[n_2]}\times\Pic_\beta(S)\big)$-valued polynomial expression in the variables 
$$
\xymatrix@R=10pt@C=0pt{
(\beta,\id,\gamma)  \ar@{=}[d]&\in& H^2(S)\ \oplus\ H^1(S)\otimes H^1(S)^*\ \oplus\ H^2(\Pic_\beta(S)) \ar@{=}[d]<-12ex> \\
c_1(\cL) & \in & \hspace{-4cm} H^2(\Pic_\beta(S)\times S).}
$$
We have shown that this polynomial vanishes on an open cone of classes $\beta\gg0$ (for any $\gamma$).
%$\gammma\gg0$ where the curve class is very positive and the Poincar\'e line bundle is twisted by a very positive line bundle over $\Pic_\beta(S)$.
It therefore vanishes for all $\beta$.
\end{proof}

\begin{cor}\label{sero} For any curve class $\beta$, let $\cD \subset S\times S_\beta$ be the universal divisor. Then for $i>0$
$$c_{n_1+n_2+i}\big(R\pi_*\;\cO(\cD)-R\hom_\pi(\cI_1(-\cD),\cI_2)\big)\=0 \quad\text{on}\ S^{[n_1]}\times S^{[n_2]}\times S_\beta.$$
\end{cor}
\begin{proof}
By \cite[Lem 2.15]{DKO} we can identify the Hilbert scheme $S_\beta$ with the projective cone $\PP^*(R^2\pi_* \cL^*(K_S))$ of quotient line bundles of $R^2\pi_* \cL^*(K_S)$, in  such a way that its natural projection to $\Pic_\beta(S)$ is given by the Abel-Jacobi morphism, and  $\cO(\cD)\cong \AJ^*\cL \otimes \cO_{\PP^*}(1)$ over $S\times S_\beta$. Now substitute
$$
F:=R\pi_*\;\cL-R\hom_{\pi}(\cI_1,\cI_2\otimes\cL),\qquad M:=\cO_{\PP^*}(1)
$$
over $S^{[n_1]}\times S^{[n_2]}\times S_\beta$ into \eqref{FoM}. Each of the terms on the right hand side vanishes for any $\beta$ by Theorem \ref{park}.
\end{proof}

\begin{rmk}
This result suggests that $R\pi_*\;\cO(\cD)-R\hom_\pi(\cI_1,\cI_2(\cD))$ has the same K-theory class as an honest vector bundle of rank $n_1+n_2$ on $S^{[n_1]}\times S^{[n_2]}\times S_\beta$. We show in \cite[Equation 4.27]{GT2} that this is actually true after we pull back an affine bundle over $S^{[n_1]}\times S^{[n_2]}\times S_\beta$. Therefore its higher Chern classes are zero after pulling back to this affine bundle. Since this pullback is an isomorphism on Chow groups \cite[Corollary 2.5.7]{Kr}, this gives another explanation for the vanishing of Corollary \ref{sero}.

Aravind Asok kindly pointed out that it is possible that any bundle on the affine bundle is pulled back from the base; this would prove $R\pi_*\;\cO(\cD)-R\hom_\pi(\cI_1,\cI_2(\cD))$ is represented by a bundle on $S^{[n_1]}\times S^{[n_2]}\times S_\beta$.
\end{rmk}

\section{Alternative approach to the virtual cycle using $\Jac(S)$}\label{jack}

Instead of removing $H^1(\cO_S)$ by hand, as we did in Section \ref{arb}, we can do it geometrically by replacing the moduli space $S^{[n]}$ of ideal sheaves by the moduli space $S^{[n]}\times\Jac(S)$ of rank 1 torsion free sheaves.

Let $\cL$ be a Poincar\'e line bundle over $S\times\Jac(S)$, and let
$$
\cL_1,\,\cL_2\To\big[S^{[n_1]}\times\Jac(S)\big]\times\big[S^{[n_2]}\times\Jac(S)\big]\times S
$$
be $\pi_{25}^*\;\cL$ and $\pi_{45}^*\;\cL$ respectively, where $\pi_{ij}$ is projection to the product of the $i$th and $j$th factors.

Then the degeneracy locus of the 2-term complex\footnote{It is only 2-term if $p_g(S)=0$. If $p_g(S)>0$ then we can pull back to an affine bundle where $H^2(\cO_S)$ splits off, as in Section \ref{splitpr}.}
\beq{2turm}
R\hom_{\pi}(\cI_1\otimes\cL_1,\cI_2\otimes\cL_2)
\eeq
is
$$
S^{[n_1,n_2]}\times\Jac(S)\ \subset\ \big[S^{[n_1]}\times\Jac(S)\big]\times\big[S^{[n_2]}\times\Jac(S)\big],
$$
where the map is the product of the usual inclusion $S^{[n_1,n_2]}\subset S^{[n_1]}\times S^{[n_2]}$ with the diagonal map $\Jac(S)\subset\Jac(S)\times\Jac(S)$.

Therefore, just as in Sections \ref{degloc} and \ref{00},
$S^{[n_1,n_2]}\times\Jac(S)$ inherits a perfect obstruction theory
$$
\big(\ext^1_p(\cI_1,\cI_2)\big)^*\To\Omega_{S^{[n_1]}\times\Jac(S)\times S^{[n_2]}\times\Jac(S)}\big|_{S^{[n_1,n_2]}\times\Jac(S)}
$$
(note the $\cL_i$ cancel over the diagonal $\Jac(S)$). And the resulting virtual cycle, pushed forward to 
$S^{[n_1]}\times\Jac(S)\times S^{[n_2]}\times\Jac(S)$, is
$$
c\_{n_1+n_2+g}\big(R\hom_\pi(\cI_1\otimes\cL_1,\cI_2\otimes\cL_2)\big),\qquad g:=h^{0,1}(S).
$$
Everything so far has been invariant under the obvious diagonal action of $\Jac(S)$. Taking a slice by pulling back to $\{\cO_S\}\times\Jac(S)\subset\Jac(S)\times\Jac(S)$ gives the following.

\begin{prop} \label{night} There is a perfect obstruction theory
\beq{mapp}
\big(\ext^1_p(\cI_1,\cI_2)\big)^*\To\Omega_{S^{[n_1]}\times S^{[n_2]}\times\Jac(S)}\big|_{S^{[n_1,n_2]}\times\{\cO_S\}}
\eeq
on $S^{[n_1,n_2]}$. The push forward of the resulting virtual cycle
$$
\big[S^{[n_1,n_2]}\big]^{\vir}\,\in\,A_{n_1+n_2}\big(S^{[n_1,n_2]}\big)
$$
to $S^{[n_1]}\times S^{[n_2]}\times\Jac(S)$ is
\beq{PD}
c_{n_1+n_2+g}\big(R\hom_{\pi}(\cI_1,\cI_2\otimes\cL)[1]\big). \vspace{-6mm}
\eeq
\ \hfill$\square$
\end{prop}

%\begin{rmk} Using the affine bundle of Section \ref{arb} to split off $H^2(\cO_S)$ only, this result can also be extended verbatim to the case $p_g(S)>0$. \end{rmk}

\begin{rmk} The canonical section
$$
\cO\To\hom(\cI_1,\cI_2)\To R\hom(\cI_1,\cI_2)
$$
over $S^{[n_1,n_2]}\times S$ gives
\beq{mapq}
R^1p_*\;\cO\To\ext^1_p(\cI_1,\cI_2).
\eeq
Dualising gives
$$
\big(\ext^1_p(\cI_1,\cI_2)\big)^*\To H^1(\cO_S)^*\otimes\cO_{S^{[n_1,n_2]}}\ \cong\ \Omega_{\Jac(S)}.
$$
One can show that this map is the projection of \eqref{mapp} to $\Omega_{\Jac(S)}$.

So letting $\ext^1_p(\cI_1,\cI_2)_0$ denote the cokernel of the injection \eqref{mapq}, we can simplify the perfect obstruction theory \eqref{mapp} to
$$
\big(\ext^1_p(\cI_1,\cI_2)\_0\big)^*\To\Omega_{S^{[n_1]}\times S^{[n_2]}}\big|_{S^{[n_1,n_2]}},
$$
recovering the one of Section \ref{arb} by Corollary \ref{vtb}.
\end{rmk}

\begin{rmk}
The degeneracy locus $S^{[n_1,n_2]}$ of Proposition \ref{night} lies in
\beq{innn}
S^{[n_1]}\times S^{[n_2]}\times\{\cO_S\}\ \stackrel j\Into\  S^{[n_1]}\times S^{[n_2]}\times\Jac(S),
\eeq
and \eqref{PD} gives an expression for the pushforward of the virtual cycle to the right hand side of \eqref{innn}. It would be nice to deduce a similar expression for the pushforward of the virtual cycle to the left hand side of \eqref{innn} (as we managed in Theorem \ref{aim} using the \emph{ad hoc} method of Section \ref{splitpr} to remove $H^1(\cO_S)$). The more geometric method of this Section does not seem to give such an expression directly. But we \emph{can} deduce it from \eqref{PD} if we use
 the generalised Carlsson-Okounkov vanishing result of Theorem \ref{park}. This allows us to write
\begin{multline} \label{mult}
c_{n_1+n_2+g}\big(R\hom_{\pi}(\cI_1,\cI_2\otimes\cL)[1]\big)\= \\
c_g\big(R\pi_*\;\cL[1]\big)\cdot
c_{n_1+n_2}\big(R\pi_*\;\cL-R\hom_{\pi}(\cI_1,\cI_2\otimes\cL)\big)
\end{multline}
on $S^{[n_1]}\times S^{[n_2]}\times\Jac(S)$, because the higher Chern classes of $R\pi_*\;\cL-R\hom_{\pi}(\cI_1,\cI_2\otimes\cL)$ vanish. (The lower Chern classes do not feature because they are multiplied by
$c_{>g}\big(R\pi_*(\cL)\big)$ which are pulled back from $\Jac(S)$ of dimension $g$ and so are zero.) 

Setting $n_1=0=n_2$ in \eqref{PD} shows $c_g\big(R\pi_*\;\cL[1]\big)$ is Poincar\'e dual to the origin $\cO_S\in\Jac(S)$ (all multiplied by $S^{[n_1]}\times S^{[n_2]}$). Since $\cL$ and $R\pi_*\;\cL$ become trivial on this locus, the right hand side of \eqref{mult} becomes
$$
j_*\,c_{n_1+n_2}\big(R\hom_{\pi}(\cI_1,\cI_2)[1]\big),
$$
using the pushforward map \eqref{innn}. Combined again with \eqref{PD} this recovers the result of Theorem \ref{aim}, that the virtual cycle's pushforward to $S^{[n_1]}\times S^{[n_2]}$ is $c_{n_1+n_2}\big(R\hom_{\pi}(\cI_1,\cI_2)[1]\big)$. This argument would only not be circular, however, if we could prove the generalised Carlsson-Okounkov vanishing of Theorem \ref{park} without using Theorem \ref{aim}.
\end{rmk}

\bibliographystyle{halphanum}
\bibliography{references}

\bigskip \noindent {\tt{amingh@math.umd.edu}} \medskip

\noindent Department of Mathematics \\
\noindent University of Maryland \\
College Park, MD 20742 \\
USA

\bigskip \noindent {\tt{richard.thomas@imperial.ac.uk}} \medskip

\noindent Department of Mathematics \\
\noindent Imperial College London\\
\noindent London SW7 2AZ \\
\noindent UK

\end{document}